\documentclass[preprint,12pt]{elsarticle}
\usepackage{amsmath,amssymb,amsthm,mathtools}
\usepackage{algorithm}
\usepackage{algpseudocode}

\usepackage{multicol}
\usepackage{array}
\usepackage{enumerate}
\usepackage{bm}
\usepackage{longtable}

\usepackage{graphicx}
\usepackage{hyperref}
\hypersetup{
	colorlinks=true,
	linkcolor=cyan,
	citecolor=cyan,
}
\usepackage{exscale}
\usepackage{breqn}
\usepackage{latexsym}
\usepackage{xspace}
\usepackage{array}
\usepackage{subcaption}
\newtheorem{theorem}{Theorem}[section]
\newtheorem{prop}{Proposition}[section]
\newtheorem{lemma}{Lemma}[section]
\newtheorem{cor}{Corollary}[theorem]

\newtheorem{de}{Definition}[section]

\newtheorem{con}{Conjecture}[section]
\newcommand{\C}{\mathcal{C}_\nu}

\newcommand{\D}{\mathcal{D}}
\newcommand{\Q}{\mathcal{Q}}
\begin{document}
	\sloppy
	\begin{frontmatter}
		
		\title{Inverses of $r$-primitive $k$-normal elements over finite fields}
		\author[1]{Mamta Rani}
		\ead{mamta11singla@gmail.com}
		\author[1]{Avnish K. Sharma}
		\ead{avkush94@gmail.com}
		\author[2]{Sharwan K. Tiwari\corref{cor1}}
		\ead{shrawant@gmail.com}
		\cortext[cor1]{Corresponding author}
		\author[1]{Anupama Panigrahi}
		\ead{anupama.panigrahi@gmail.com }
		\address[1]{Department of Mathematics, University of Delhi, New Delhi, 110007, India}
		\address[2]{Scientific Analysis Group, Defence Research $\&$ Development Organization, Metcalfe House, Delhi, 110054, India}
\begin{abstract}
	Let $r$, $n$ be positive integers, $k$ be a non-negative integer and $q$ be any prime power such that $r\mid q^n-1.$ An element $\alpha$ of the finite field $\mathbb{F}_{q^n}$ is called an {\it $r$-primitive} element, if its multiplicative order is $(q^n-1)/r$,  and it is called a {\it $k$-normal} element over $\mathbb{F}_q$, if the greatest common divisor of the polynomials $m_\alpha(x)=\sum_{i=1}^{n} \alpha^{q^{i-1}}x^{n-i}$ and $x^n-1$ is of degree $k.$ In this article, we define the characteristic function for the set of $k$-normal elements, and with the help of this, we establish a sufficient condition for the existence of an element $\alpha$ in $\mathbb{F}_{q^n}$, such that $\alpha$ and $\alpha^{-1}$ both are simultaneously $r$-primitive and $k$-normal over $\mathbb{F}_q$. Moreover, for $n>6k$, we show that there always exists an $r$-primitive and $k$-normal element $\alpha$ such that $\alpha^{-1}$ is also $r$-primitive and $k$-normal in all but finitely many fields $\mathbb{F}_{q^n}$ over $\mathbb{F}_q$, where $q$ and $n$ are such that $r\mid q^n-1$ and there exists a $k$-degree polynomial $g(x)\mid x^n-1$ over $\mathbb{F}_q$. In particular, we discuss the existence of an element $\alpha$ in $\mathbb{F}_{q^n}$ such that $\alpha$ and $\alpha^{-1}$ both are simultaneously $1$-primitive and $1$-normal over $\mathbb{F}_q$. 
\end{abstract}
		\end{frontmatter}
	\section{Introduction}\label{Sec1}
	Let $\mathbb{F}_{q^n}$ be the field extension of degree $n$ over $\mathbb{F}_{q},$ where $q$ be a prime power and $n\in \mathbb{N}$. We recall that, the multiplicative group $\mathbb{F}_{q^n}^*$ is cyclic, and an element $\alpha \in \mathbb{F}_{q^n}^*$ is called {\it primitive}, if its multiplicative order is $q^n-1.$ Let $r$ be a divisor of $q^n-1$, then an element $\alpha\in \mathbb{F}_{q^n}^*$ is called {\it $r$-primitive}, if its multiplicative order is $(q^n-1)/r$. Clearly, a 1-primitive element in $\mathbb{F}_{q^n}$ is actually a primitive element. Moreover, the existence of $r$-primitive elements is obvious, because the multiplicative order of the $r^{\mathrm{th}}$ power of a primitive element is  $(q^n-1)/r$. Primitive elements have many applications in the field of cryptography \cite{diffie,mullennote,blum} due to their highest multiplicative order.  Moreover, $r$-primitive elements can  also be considered as high ordered elements for small values of $r$.  Therefore, in many applications, $r$-primitive elements may replace primitive elements.
	
	Again, recall that, the finite field $\mathbb{F}_{q^n}$ is also an $\mathbb{F}_{q}$-vector space, and an element $\alpha \in \mathbb{F}_{q^n}$ is said to be {\it normal} over $\mathbb{F}_{q}$, if the set $\{\alpha, \alpha^q, \ldots, \alpha^{q^{n-1}}\}$ forms a basis (called {\it normal basis}) of $\mathbb{F}_{q^n}$ over $\mathbb{F}_{q}$. Normal bases have many applications in the computational theory due to their efficient  implementation in finite field arithmetic. For more detail on normal elements one may refer \cite{gao} and the references therein. 
	
	From \cite[Theorem 2.39]{Nieder}, we know that an element $\alpha \in \mathbb{F}_{q^n}$ is normal over $\mathbb{F}_{q}$ if and only if the polynomials $m_\alpha(x)=\alpha x^{n-1}+\alpha^q x^{n-2}+\ldots+\alpha^{q^{n-2}}x+\alpha^{q^{n-1}}$ and $x^n-1$ are relatively prime in $\mathbb{F}_{q^n}[x].$ Motivated by this definition, Huczynska et al. \cite{huc13} introduced the concept of $k$-normal elements and gave the following definition.
	\begin{de}
		An element $\alpha \in \mathbb{F}_{q^n}$ is said to be {\it $k$-normal} over $\mathbb{F}_{q},$ if the polynomials $m_\alpha(x)=\alpha x^{n-1}+\alpha^q x^{n-2}+\ldots+\alpha^{q^{n-2}}x+\alpha^{q^{n-1}}$ and $x^n-1$ has greatest common divisor of degree $k$ over $\mathbb{F}_{q}$.
	\end{de}
	 \noindent
	 From the above definition, it is clear that $0\leq k\leq n-1.$ We call an element $\alpha\in \mathbb{F}_{q^n}$, {\it $r$-primitive $k$-normal}, if it is both $r$-primitive and $k$-normal over $\mathbb{F}_{q}.$ From the above definition, if $k=0$, then $\alpha$ is normal over $\mathbb{F}_{q}$. Moreover, Huczynska et al. \cite{huc13} characterized the $k$-normal elements in $\mathbb{F}_{q^n}$ over $\mathbb{F}_q$ and established bounds for the number of such elements, and proved the existence of an element $\alpha$ which is both $1$-primitive and $1$-normal over $\mathbb{F}_{q}$ for $n\geq 6$ if $q\geq 11$, and for $n\geq 3$ if $3\leq q\leq 9,$ where $\mathrm{gcd}(q,n)=1.$
	 In 1987, Lenstra and Schoof \cite{Lenstra} proved the existence of primitive normal elements (i.e. $1$-primitive $0$-normal elements) for arbitrary finite fields. In 2014, Anderson and Mullen \cite{mullConj} conjectured that, for a prime $p\geq 5$ and an integer $n\geq3$, there always exists an element $\alpha\in \mathbb{F}_{p^n}$ of order $(p^n-1)/r$, which is $k$-normal for $(r,k)=(1,1), (2,0), (2,1)$. In 2018, Reis and Thomson in \cite{reis18} settled their conjecture for $(r,k)=(1,1)$, for arbitrary $q$ and $n\geq3$. Finally, in 2019, Kapetanakis and Reis \cite{kapevar} completely settled the above conjecture for arbitrary prime power $q$. 
	 
	 Recently, Aguirre and Neumann \cite{newman} showed that, there exists a $1$-primitive $2$-normal element in $\mathbb{F}_{q^n}$ over $\mathbb{F}_q$ if and only if $n\geq5$ and $\mathrm{gcd}(q^3-q, n) \neq 1$ or $n =4$ and $q \equiv 1\hspace{.5mm} (\mathrm{mod}\hspace{.5mm} 4).$ In this direction, we \cite{Rani2} generalized the above existence problems to arbitrary $r$ and $k$, and established a sufficient condition for the existence of an element $\alpha$ in $\mathbb{F}_{q^n}$ which is simultaneously $r$-primitive and $k$-normal over $\mathbb{F}_q.$ Moreover, in particular, we gave the following complete existence of $2$-primitive $2$-normal elements.
	\begin{theorem}\label{T1.1}
		Let $q$ be a prime power and $n$ be a positive integer. Then, there always exists a $2$-primitive $2$-normal element in $\mathbb{F}_{q^n}$ over $\mathbb{F}_q$ if and only if $q$ is odd, and either $n\geq 5$ and $\mathrm{gcd}(q^3-q,n)\neq1$ or $n=4$ and  $q\equiv1(\mathrm{mod}\hspace{.5mm}4)$.
	\end{theorem}
	
	Notice that, the inverse of an $r$-primitive element is also $r$-primitive in $\mathbb{F}_{q^n}$ for $r\mid q^n-1.$ But this may not be true in case of $k$-normal elements. For example, if $k=0$, then there does not exist any normal element $\alpha\in \mathbb{F}_8$ over $\mathbb{F}_2$ for which $\alpha^{-1}$ is also normal element in $\mathbb{F}_{2^3}$ over $\mathbb{F}_2$.
	In \cite{TianQi}, Tian and Qi showed the existence of a primitive normal element $\alpha\in \mathbb{F}_{q^n}$ over $\mathbb{F}_q$ whose inverse is also primitive normal over $\mathbb{F}_q$ for all $q$ and $n\geq 32.$ Later, Cohen and Huczynska \cite{CoHuc} completely settled this problem by using sieving technique and gave the following result.
	\begin{theorem}[Strong Primitive Normal Basis Theorem] There exists a $(1-)$ primitive $(0-)$ normal element $\alpha \in \mathbb{F}_{q^n};$ $n\geq 2,$ such that $\alpha^{-1}$ is also a primitive normal element over $\mathbb{F}_{q}$ unless $(q,n)$ is one of the pair $(2,3),\ (2,4),\ (3,4),\ (4,3),\ (5,4).$	
	\end{theorem}
	Motivated by the above theorem, in this article, we discuss the existence of a pair $(\alpha,\alpha^{-1})$ of $r$-primitive $k$-normal elements in $\mathbb{F}_{q^n}$ over $\mathbb{F}_q$. For this, we first define the characteristic function in Section \ref{Sec3} for the set of $k$-normal elements, which plays a crucial role in establishing a sufficient condition for the existence of such pairs in Section \ref{Sec4}. Moreover, we conclude that, for $n>6k$, there exists a pair $(\alpha,\alpha^{-1})$ of  $r$-primitive $k$-normal elements in all but finitely many fields $\mathbb{F}_{q^n}$ over $\mathbb{F}_q.$ In particular, we gave the following result in Section \ref{Sec5}.
		\begin{theorem}\label{T1.3}
			Let $q$ be a prime power and $n\geq 5$ be an integer. Then	
			\begin{enumerate}[$(i)$]
				\item For $n\geq 7$ and $q\geq 2$, there always exists a pair $(\alpha, \alpha^{-1})$ of primitive $1$-normal elements in $\mathbb{F}_{q^n}$ over $\mathbb{F}_q$.
				\item For $n=5$, $6$ and $q\geq 2$ such that $\mathrm{gcd}(q,n)=1$, there always exists a pair $(\alpha, \alpha^{-1})$ of primitive $1$-normal elements in $\mathbb{F}_{q^n}$ over $\mathbb{F}_q$ with the sole genuine exception $(4,5).$  
		    \end{enumerate}
		\end{theorem}
	
\section{Preliminaries}\label{Sec2}
	In this section, we recall some definitions and lemmas that are crucial for proving the results of this article.
	\subsection{Freeness}
	We begin with the following definition.
	\begin{de}[$e$-free element]
		Let $e\mid q^n-1$ and $\beta \in \mathbb{F}_{q^n}.$ Then $\beta$ is said to be $e$-free, if $\beta$ is not a $d^{th}$ power in $\mathbb{F}_{q^n}$ for any divisor $d>1$ of $e$. In other words, $\beta$ is $e$-free if and only if $e$ and $\frac{q^n-1}{\mathrm{ord}(\beta)}$ are co-prime, where $\mathrm{ord}(\beta)$ is the multiplicative order of $\beta$ in $\mathbb{F}_{q^n}^*.$ Clearly, an element $\beta \in \mathbb{F}_{q^n}$ is primitive if and only if it is $(q^n-1)$-free. 
	\end{de}
	\noindent
	 Let $h(x)=\sum_{i=0}^{m}a_ix^i\in\mathbb{F}_{q}[x]$ with $a_m\neq0$. Then, the additive group $\mathbb{F}_{q^n}$ forms an $\mathbb{F}_{q}[x]$- module under the action $h\circ\beta=\sum_{i=0}^{m}a_i\beta^{q^i}$. Notice that, $(x^n-1)\circ\beta=0$ for all $\beta\in \mathbb{F}_{q^n},$ which leads to the following definition.
	\begin{de}
		The $\mathbb{F}_q$-order $($denoted by $\mathrm{Ord}_q(\cdot))$ of an element $\beta\in \mathbb{F}_{q^n},$ is the least degree monic divisor $h(x)$ of $x^n-1$ such that $h\circ\beta=0$.
	\end{de}
	\noindent
	Similar to the $e$-free elements, we have the following definition for $f$-free elements.
	\begin{de}
		Let $h(x)\mid x^n-1$ and $\beta \in \mathbb{F}_{q^n}.$ Then, $\beta$ is called $h$-free, if $\beta\neq g\circ \gamma$ for any  $\gamma  \in \mathbb{F}_{q^n}$ and for any divisor $g(x)$ of $h(x)$ of positive degree. In other words, $\beta$ is $h$-free if and only if $h(x)$ and $\frac{x^n-1}{\mathrm{Ord}_q(\beta)}$ are co-prime.   Clearly, an element $\beta \in \mathbb{F}_{q^n}$ is normal if and only if it is $(x^n-1)$-free. 
	\end{de}

	In this article, we will use the following definition of $k$-normal elements.
	\begin{de}{\upshape\cite[Theorem 3.2]{huc13}}
		An element $\beta\in \mathbb{F}_{q^n}$ is $k$-normal over $\mathbb{F}_q$, if and only if the $\mathbb{F}_q$-order of $\beta$ is of degree $n-k$.
	\end{de}
\noindent
	The following lemma provides a way to construct the $k$-normal elements from a given normal element.
	\begin{lemma}{\upshape\cite[Lemma 3.1]{reis19}}\label{L2.1}
		Let $\beta\in\mathbb{F}_{q^n}$ be a normal element over $\mathbb{F}_q$ and $g\in \mathbb{F}_q[x]$ be
		a polynomial of degree $k$ such that $g$ divides $x^n-1$. Then $\alpha=g\circ\beta$ is
		$k$-normal.
	\end{lemma} 
	\subsection{Characters}
	Let $\mathfrak{G}$ be a finite abelian group. A {\it character} $\chi$ of $\mathfrak{G}$ is a homomorphism from $\mathfrak{G}$ into the multiplicative group of  complex numbers of unit modulus. The set of all such characters of $\mathfrak{G}$, denoted by $\widehat{\mathfrak{G}}$, forms a multiplicative group and $\mathfrak{G}\cong\widehat{\mathfrak{G}}$. A character $\chi$ is called the trivial character if $\chi(\alpha)=1$ for all $\alpha\in \mathfrak{G}$, otherwise it is a non-trivial character.
	\begin{lemma}{\upshape\cite[Theorem 5.4]{Nieder}}\label{L2.2}
		Let $\chi$ be any non-trivial character of a finite abelian group $\mathfrak{G}$ and $\alpha \in \mathfrak{G}$ be any non-trivial element, then
		$$\sum_{\alpha \in \mathfrak{G}} \chi(\alpha)=0 \ \text{and} \ \sum_{\chi \in \widehat{\mathfrak{G}}} \chi(\alpha)=0.$$
	\end{lemma}
	Let $\psi$ denote the additive character for the additive group $\mathbb{F}_{q^n}$, and $\chi$ denote the multiplicative character for the multiplicative group $\mathbb{F}_{q^n}^*.$
	The additive character $\psi_0$ defined by $\psi_{0}(\beta)=e^{2\pi i \mathrm{Tr}(\beta)/p}, \ \text{for all} \ \beta \in \mathbb{F}_{q^n},$
	where $p$ is the characteristic of $\mathbb{F}_{q^n}$ and $\mathrm{Tr}$ is the absolute trace function from $\mathbb{F}_{q^n}$ to    $\mathbb{F}_p$, is called the {\it canonical additive character} of $\mathbb{F}_{q^n}$. Moreover, every additive character $\psi_\beta$ for $\beta \in \mathbb{F}_{q^n}$ can be expressed in terms of the canonical additive character $\psi_0$ as $\psi_\beta(\gamma)=\psi_{0}(\beta\gamma),\ \text{for all} \ \gamma \in \mathbb{F}_{q^n}.$ 
	
	For any $\psi \in \widehat{\mathbb{F}}_{q^n}$, $\alpha \in \mathbb{F}_{q^n}$ and $g(x)\in \mathbb{F}_q[x]$, $\widehat{\mathbb{F}}_{q^n}$ is an $\mathbb{F}_{q}[x]$-module under the action $\psi\circ g(\alpha)=\psi(g\circ \alpha).$ The {\it $\mathbb{F}_q$-order} of an additive character $\psi\in \widehat{\mathbb{F}}_{q^n}$, denoted by $\mathrm{Ord}_q(\psi)$, is the least degree monic divisor $g(x)$ of $x^n-1$ such that $\psi\circ g$ is the trivial character and there are precisely $\Phi_q(g)$ characters of $\mathbb{F}_q$-order $g(x)$, where $\Phi_q(g)=|(\mathbb{F}_{q}[x]/<g>)^*|$. Moreover, $\sum_{h\mid g}\Phi_q(h)=q^{\mathrm{deg}(g)}.$
	
	We shall need the following lemma to prove our sufficient condition.
	\begin{lemma}\label{L2.3}{\upshape\cite[Theorem 5.5]{Fu}}
		Let $F(x)=\prod_{j=1}^{k}f_j(x)^{n_j}$ be a rational function in $\mathbb{F}_{q^n}(x)$, where $f_j(x)\in \mathbb{F}_{q^n}[x]$ are irreducible polynomials and $n_j$ are non-zero integers. Let $\chi$ be a multiplicative character of order $d$ of $\mathbb{F}_{q^n}^*.$ Suppose that the rational function $F(x)$ is not of the form $L(x)^{d}$ for any $L(x)\in \mathbb{F}(x),$ where $\mathbb{F}$ is the algebraic closure of $\mathbb{F}_{q^n}.$ Then we have  
		$$\Big |\sum_{\substack{\alpha \in \mathbb{F}_{q^n}\\F(\alpha)\neq 0,\infty}}\chi(F(\alpha))\Big|\leq \Big(\sum_{j=1}^{k}\mathrm{deg}(f_j)-1\Big)q^{n/2}.$$
	\end{lemma}
	\begin{lemma}\label{L2.4}{\upshape\cite[Theorem 5.6]{Fu}}
		Let $F(x),\ G(x) \in \mathbb{F}_{q^n}(x)$ be rational functions. Write $F(x)=\prod_{j=1}^{k}f_j(x)^{n_j}$, where $f_j(x)\in \mathbb{F}_{q^n}[x]$ are irreducible polynomials and $n_j$ are non-zero integers. Let $d_1=\sum_{j=1}^{k}\mathrm{deg}(f_j(x))$, $d_2=max\{\mathrm{deg}(G(x)),0\}$, $d_3$ be the degree of the denominator of $G(x)$ and $d_4$ be the sum of degrees of those irreducible polynomials dividing the denominator of $G(x)$, but distinct from $f_j \ ;\ 1\leq j\leq k.$ Let $\chi$ be the multiplicative character of $\mathbb{F}_{q^n}^*$ and $\psi$ be the non-trivial additive character of $\mathbb{F}_{q^n}.$ Suppose that $G(x)\neq L(x)^{q^n}-L(x)$ in $\mathbb{F}(x),$ where $\mathbb{F}$ is the algebraic closure of $\mathbb{F}_{q^n}.$ Then we have 
		$$\Big |\sum_{\substack{\alpha \in \mathbb{F}_{q^n},F(\alpha)\neq 0,\infty \\ G(\alpha)\neq \infty}} \chi(F(\alpha))\psi(G(\alpha))\Big|\leq (d_1+d_2+d_3+d_4-1)q^{n/2}.$$ 
	\end{lemma}
	\subsection{Some characteristic functions} 
	For the set of $e$-free elements of $\mathbb{F}_{q^n}^*$, we have the following characteristic function $\rho_e:\mathbb{F}_{q^n}^*\to \{0,1\}.$
	\begin{equation}\label{Eq1}
	\rho_{e}(\beta)= \frac{\phi(e)}{e}\sum_{d|e}\frac{\mu(d)}{\phi(d)}\sum_{\chi_{d}}\chi_{d}(\beta),
	\end{equation}
	where $\mu$ is the M\"obius function and the internal sum runs over all the multiplicative characters $\chi_d$ of order $d$.  Similar to the characteristic function for the set of $e$-free elements, we have the following characteristic function $\Upsilon_g:\mathbb{F}_{q^n}\to \{0,1\}$ for the set of $g$-free elements of $\mathbb{F}_{q^n}$.
	\begin{equation}\label{Eq2}
	\Upsilon_{g}(\alpha)= \dfrac{\Phi_q(g)}{q^{\mathrm{deg}(g)}}\sum_{h|g}\dfrac{\mu'(h)}{\Phi_{q}(h)}\sum_{\psi_{h}}\psi_{h}(\alpha),
	\end{equation}
	where $\mu'$ is the analog of the M\"obius function, which is defined as $\mu'(h)=(-1)^t$, if $h(x)$ is a product of $t$ distinct monic irreducible polynomials, otherwise 0, and the internal sum runs over the additive characters $\psi_{h}$ of the $\mathbb{F}_{q}$-order $h(x)$.
	
	\section{Characterization of $r$-primitive and $k$-normal elements}\label{Sec3}
	 
	\subsection{Characteristic function for the set of $r$-primitive elements}
	Let $r\mid q^n-1$ be a positive integer, and write $r=up_1^{b_1}p_2^{b_2}\cdots p_s^{b_s}$, where $u$ and $(q^n-1)/u$ are co-prime and $p_j$'s are distinct primes such that $b_j\geq1$ and $p_j^{b_j+1}\mid q^n-1$, for all $j=1,2,\ldots,s.$ In this article, we denote the product of distinct irreducible factors of an integer or a polynomial $m$ by $\mathrm{rad}(m)$. Now, let $R=\mathrm{rad}(q^n-1)/\mathrm{rad}(r)$, and set $\delta_j:=p_j^{b_j}$ and $\lambda_j:=p_j^{b_j+1}$ for all $j=1,2,\ldots,s$. In \cite{cohenrprim}, Cohen and Kapetanakis showed that an element $\alpha \in \mathbb{F}_{q^n}$ is $r$-primitive if and only if $\alpha$ is $R$-free, $\alpha=\beta^r$ for some $\beta \in \mathbb{F}_{q^n}$, but $\alpha\neq\beta^{\lambda_j}$ for any $\beta \in \mathbb{F}_{q^n}$ and any $j=1,2,\ldots,s$, and using this fact, they defined the characteristic function for $r$-primitive elements. In \cite{Rani2}, following them, we defined the following characteristic function $\Gamma_r^d$ for the set $Q_r^d$ of elements $\alpha\in \mathbb{F}_{q^n}$ such that $\alpha$ is $d$-free for some divisor $d$ of $R$, $\alpha=\beta^r$ for some $\beta \in \mathbb{F}_{q^n}$, but $\alpha\neq\beta^{\lambda_j}$ for any $\beta \in \mathbb{F}_{q^n}$ and any $j=1,2,\ldots,s.$
	\begin{equation}\label{Eq3}
	\Gamma_r^d(\alpha)=\frac{\phi(d)}{rd}\sum_{\substack{d_1\mid d,d_2\mid u}}\sum_{\substack{e_j\mid \lambda_j,\\1\leq j\leq s}}\frac{\mu(d_1)}{\phi(d_1)}\big(\prod_{j=1}^s\ell_{p_j,e_j}\big)\sum_{\substack{\chi_{d_1},\chi_{d_2},\\\chi_{e_1},\ldots,\chi_{e_s}}}(\chi_{d_1}\chi_{d_2}\chi_{e_1}\cdots\chi_{e_s})(\alpha),
	\end{equation}  
	where 
	$\ell_{p_j,e_j}=\left\{\begin{array}{ll}
	1-1/p_j&\mathrm{if}\ e_j\neq \lambda_j,\\
	-1/p_j&\mathrm{if}\ e_j=\lambda_j
	\end{array}\right..$
	In particular, for $d=R$, the above characteristic function provides the characteristic function given by Cohen and Kapetanakis \cite{cohenrprim} for the set $Q_r^R$ of $r$-primitive elements.
	\subsection{Characteristic function for the set of $k$-normal elements}
	For the existence of those elements in a finite field $\mathbb{F}_{q^n}$ that simultaneously satisfy certain properties, such as primitivity, normality, trace, norm, etc., a general approach is to use the characteristic functions for the set of elements with those properties. In this article, we also use this approach to discuss the existence of a pair $(\alpha,\alpha^{-1})$ of $r$-primitive $k$-normal elements in $\mathbb{F}_{q^n}$ over $\mathbb{F}_q$. For this, we shall need the characteristic function for the set of $k$-normal elements that we define in this subsection.
	
	Let $S_k$ be the collection of all $k$-normal elements in $\mathbb{F}_{q^n}$ over $\mathbb{F}_q$ and $P_k$ be the collection of polynomials of degree $k$ that divides $x^n-1$.
	By Lemma \ref{L2.1}, if $\beta$ is a normal element of $\mathbb{F}_{q^n}$ over $\mathbb{F}_q$, then for any polynomial $g\in P_k$, $g\circ\beta$ is $k$-normal over $\mathbb{F}_q$. So, for every $g\in P_k$, we consider a set $S_{g,k}:=\{\alpha \in S_k\ |\ \alpha=g\circ\beta \text{ for some element}\ \beta\in\mathbb{F}_{q^n}\}$ of $k$-normal elements. We have the following lemma.
	\begin{lemma}\label{L3.1}
		If $g,\ g'$ are two distinct polynomials in $P_k$, then  $S_{g,k}\bigcap S_{g',k}=\emptyset$, and $\bigcup_{g\in P_k}S_{g,k}=S_k$.
	\end{lemma} 
	\begin{proof}
		Suppose that, $S_{g,k}\bigcap S_{g',k}\neq\emptyset$, then there exist elements $\beta$ and $\beta'$ in $\mathbb{F}_{q^n}$ such that $g\circ\beta=g'\circ\beta'$, and by uniqueness of the $\mathbb{F}_q$-order of an element, we get $g=g'$, a contradiction. Hence, $S_{g,k}\bigcap S_{g',k}=\emptyset$.
		
		Clearly, $\bigcup_{g\in P_k}S_{g,k}\subseteq S_k.$ Now, let $\alpha\in S_k$, then $\alpha=h\circ\beta$, for some normal element $\beta$ and a polynomial $h\in\mathbb{F}_q[x]$ of degree at most $n-1$. Let $f$ be the $\mathbb{F}_q$-order of $\alpha$, then $f\circ\alpha=fh\circ\beta=0$, and hence, $x^n-1\mid fh$, i.e. $h=(x^n-1)h'/f$, for some $h'\in\mathbb{F}_q[x].$ Thus, $\alpha=\frac{x^n-1}{f}\circ(h'\circ\beta)$, which implies $\alpha=g\circ\gamma\in S_{g,k}$, where $g=\frac{x^n-1}{f}\in P_k$ and $\gamma=h'\circ\beta\in \mathbb{F}_{q^n}$. Hence, $\bigcup_{g\in P_k}S_{g,k}=S_k$.   
	\end{proof}

	From the above lemma, it is clear that, the set $\{S_{g,k}\ | \ g\in P_k\}$ provides a partition of $S_k$. Therefore, the characteristic function for the set $S_k$ of $k$-normal elements is equal to the sum of the characteristic functions for the sets $S_{g,k}\hspace{.5mm};\ g\in P_k.$ Hence, it is enough to define the characteristic function for the set $S_{g,k}$.

		Let $g\in P_k$ be an arbitrary but fixed polynomial. Write $g=\pi f_1^{b_1}f_2^{b_2}\cdots f_t^{b_t}$, where $\pi$ and $(x^n-1)/\pi$ are co-prime, and $f_i$'s are distinct irreducible polynomials such that $b_i\geq1$ and $f_i^{b_i+1}\mid x^n-1$, for all $i=1,2,\ldots,t.$ Further, let $G=\mathrm{rad}(x^n-1)/\mathrm{rad}(g)$, and set $\Delta_i:=f_i^{b_i}$ and $\Lambda_i:=f_i^{b_i+1}$ for all $i=1,2,\ldots,t.$ Now, we prove the following lemma. 
	\begin{lemma}\label{L3.2}
		Let $k$ be a non-negative integer and $g\in P_k$. Then, the following are equivalent.
		\begin{enumerate}[$(i)$]
			\item $\alpha\in S_{g,k}$.
			\item $\alpha$ is $G$-free, $\alpha=g\circ\beta$ for some $\beta\in \mathbb{F}_{q^n}$, but $\alpha\neq\Lambda_i\circ\beta$ for any $\beta\in\mathbb{F}_{q^n}$ and for any $i=1,2,\ldots,t.$  
		\end{enumerate}
	\end{lemma}
	\begin{proof}
		First, let $\alpha\in S_{g,k}$. Then by definition of $S_{g,k}$, $\alpha=g\circ\beta$ for some $\beta\in \mathbb{F}_{q^n}$, and $\mathbb{F}_q$-order of $\alpha$ is $(x^n-1)/g$. Notice that, $G$ and $g$ are co-prime, therefore, $\mathrm{gcd}\big(G,\frac{x^n-1}{\mathrm{Ord}_q(\alpha)}\big)=1$, which implies $\alpha$ is $G$-free. Now, suppose that $\alpha=\Lambda_i\circ\beta$ for some $\beta\in \mathbb{F}_{q^n}$ and some $i=1,2,\ldots,t$. Then, $\frac{x^n-1}{\Lambda_i}\circ\alpha=0$, which implies $\frac{x^n-1}{g}\mid\frac{x^n-1}{\Lambda_i}$, i.e. $\Lambda_i\mid g,$ which is not possible. Hence, $(i)$ implies $(ii)$.
		
		Now, let $\alpha$ is $G$-free, $\alpha=g\circ\beta$ for some $\beta\in \mathbb{F}_{q^n}$, but $\alpha\neq\Lambda_i\circ\beta$ for any $\beta\in\mathbb{F}_{q^n}$ and for any $i=1,2,\ldots,t.$ Then, $\alpha=g\circ(h\circ\gamma)$ for some normal element $\gamma\in\mathbb{F}_{q^n}$ and a polynomial $h\in \mathbb{F}_q[x]$ of degree at most $n-1$, and $\mathrm{Ord}_q(\alpha)=\frac{x^n-1}{g\hspace{.5mm}\mathrm{gcd}(h,(x^n-1)/g)}$. Suppose that, $h'$ be an irreducible factor of the $\mathrm{gcd}(h,(x^n-1)/g)$, then $h'\mid h$ and $h'\mid\frac{x^n-1}{g}$. This implies, either $h'\mid G$ or $\Lambda_i\mid gh.$ In any case, we get a contradiction. This means $\mathrm{Ord}_q(\alpha)=\frac{x^n-1}{g}$, i.e. $\alpha$ is a $k$-normal element. Hence, $(ii)$ implies $(i)$.  
	\end{proof}

		Now, we define the characteristic function $\Psi_g$ for the set of elements of the form $g\circ \beta.$ Consider the set $M_g$ of the elements of the form $g\circ\beta$, where $g\mid x^n-1.$ Clearly, $M_g$ is a subgroup of the additive group $\mathbb{F}_{q^n}$. Let $A\subseteq \widehat{\mathbb{F}}_{q^n}$ be the annihilator of $M_g$, i.e. the collection of all additive characters $\psi\in \widehat{\mathbb{F}}_{q^n}$ such that $\psi(\alpha)=1$ for all $\alpha\in M_g.$ Clearly, the set $A$ consists of all the additive characters of $\mathbb{F}_q$-order $h$ dividing $g$, and from \cite[Theorem 5.6]{Nieder}, it is isomorphic to the group of the characters of the quotient group $\mathbb{F}_{q^n}/M_g$, i.e. $A\cong\widehat{(\mathbb{F}_{q^n}/M_g)}$. Therefore, we can define a character $\mathcal{Y}_h$ on $\widehat{(\mathbb{F}_{q^n}/M_g)}$ by $\mathcal{Y}_h(\alpha+M_g)=\psi_h(\alpha),$ where $\psi_h\in A.$ Then, 
		$$\sum_{h\mid g}\sum_{\mathcal{Y}_h}\mathcal{Y}_h(\alpha+M_g)=\left\{\begin{array}{lll}
		q^{\mathrm{deg}(g)}&,&\mathrm{if}\ \alpha\in M_g,\\
		0&,&\mathrm{otherwise}
		\end{array}\right..$$
		Thus, we can define a characteristic function $\Psi_g:\mathbb{F}_{q^n}\to\{0,1\}$ for the set $M_g$ as follows.
	\begin{equation}\label{Eq4}
		\Psi_g(\alpha):=\frac{1}{q^{\mathrm{deg}(g)}}\sum_{h\mid g}\sum_{\psi_h}\psi_h(\alpha).
	\end{equation}
	
	Now, we are ready to define a characteristic function $\Q_g^G$ for the set $S_{g,k}.$ From the characteristic functions \eqref{Eq2} and \eqref{Eq4} and from Lemma \ref{L3.2}, we have   
	\begin{align*}
	\Q_g^G(\alpha)=\Upsilon_G(\alpha)\Psi_g(\alpha)\prod_{i=1}^t\big(1-\Psi_{\Lambda_i}(\alpha)\big).
	\end{align*}
	Since, $g=\pi\prod_{i=1}^{t}\Delta_i$, and $\pi$, $\Delta_i\ ;\ i=1,2,\ldots,t$ are mutually co-prime, $\Psi_g(\alpha)=\Psi_{\pi}(\alpha)\prod_{i=1}^t\Psi_{\Delta_i}(\alpha)$, and hence,
	\begin{align*}
	\Q_g^G(\alpha)=\Upsilon_G(\alpha)\Psi_{\pi}(\alpha)\prod_{i=1}^t\Psi_{\Delta_i}(\alpha)\big(1-\Psi_{\Lambda_i}(\alpha)\big).
	\end{align*}
	Moreover, if $\alpha=\Lambda_i\circ\beta$ for some $\beta$, then $\alpha=\Delta_i\circ\gamma$ for some $\gamma$ in $\mathbb{F}_{q^n}$, and $\Psi_{\Delta_i}\Psi_{\Lambda_i}=\Psi_{\Lambda_i}$ for all $i=1,2,\ldots,t.$ Hence,
	\begin{equation}\label{Eq5}
	\Q_g^G(\alpha)=\Upsilon_G(\alpha)\Psi_{\pi}(\alpha)\prod_{i=1}^t\big(\Psi_{\Delta_i}(\alpha)-\Psi_{\Lambda_i}(\alpha)\big).
	\end{equation}
	Now,
	\begin{align*}
	\Psi_{\Delta_i}(\alpha)-\Psi_{\Lambda_i}(\alpha)=&\frac{1}{q^{\mathrm{deg}(\Delta_i)}}\sum_{h\mid \Delta_i}\sum_{\psi_h}\psi_{h}(\alpha)-\frac{1}{q^{\mathrm{deg}(\Lambda_i)}}\sum_{h\mid \Lambda_i}\sum_{\psi_{h}}\psi_{h}(\alpha)\\
	=&\frac{1}{q^{\mathrm{deg}(\Delta_i
	)}}\sum_{h\mid\Lambda_i }\sum_{\psi_{h}}\ell_{f_i,h}'\psi_{h}(\alpha),
	\end{align*}	
	where 
	$\ell_{f_i,h}'=\left\{\begin{array}{ll}
	1-1/q^{\mathrm{deg}(f_i)}&\mathrm{if}\ h\neq \Lambda_i,\\
	-1/q^{\mathrm{deg}(f_i)}&\mathrm{if}\ h=\Lambda_i
	\end{array}\right..$
	Notice that, $|\ell_{f_i,h}'|\leq\ell_{f_i,1}'$ for all $i=1,2,\ldots,t$.
	Now, substituting the values of $\Upsilon_G(\alpha)$, $\Psi_{\pi}(\alpha)$ and $\Psi_{\Delta_i}(\alpha)-\Psi_{\Lambda_i}(\alpha)$ in Equation \eqref{Eq5}, we get
	{\small \begin{equation}\label{Eq6}
	\begin{aligned}
	\Q_g^G(\alpha)=&\frac{\Phi_q(G)}{q^{\mathrm{deg}(G)+\mathrm{deg}(g)}}\sum_{\substack{g_1\mid G,\\g_2\mid \pi}}\sum_{\substack{h_i\mid \Lambda_i,\\1\leq i\leq t}}\frac{\mu'(g_1)}{\Phi_q(g_1)}\big(\prod_{i=1}^t\ell_{f_i,h_i}'\big)\sum_{\substack{\psi_{g_1},\psi_{g_2},\\\psi_{h_i};\ 1\leq i\leq t}}(\psi_{g_1}\psi_{g_2}\psi_{h_1}\cdots\psi_{h_t})(\alpha).
	\end{aligned}
	\end{equation}}
	Thus, we define the characteristic function $\Q$ for the set $S_k$ of all $k$-normal elements as $\Q=\sum_{g\in P_k}\Q_g^G$.

	Now, for a divisor $H$ of $G$, we define a collection $T_{g,k}^H$ of elements $\alpha$ of $\mathbb{F}_{q^n}$ such that $\alpha$ is $H$-free, $\alpha=g\circ\beta$ for some $\beta\in \mathbb{F}_{q^n}$, but $\alpha\neq\Lambda_i\circ\beta$ for any $\beta\in\mathbb{F}_{q^n}$ and for any $i=1,2,\ldots,t.$ We define the characteristic function $\Q_g^H$ for the set $T_{g,k}^H$ similar to that for the set $S_{g,k}$ as follows
	{\small \begin{equation}\label{Eq7}
	\begin{aligned}
	\Q_g^H(\alpha)=&\frac{\Phi_q(H)}{q^{\mathrm{deg}(H)+\mathrm{deg}(g)}}\sum_{\substack{g_1\mid H,\\g_2\mid \pi}}\sum_{\substack{h_i\mid \Lambda_i,\\1\leq i\leq t}}\frac{\mu'(g_1)}{\Phi_q(g_1)}\big(\prod_{i=1}^t\ell_{f_i,h_i}'\big)\sum_{\substack{\psi_{g_1},\psi_{g_2},\\\psi_{h_i};\ 1\leq i\leq t}}(\psi_{g_1}\psi_{g_2}\psi_{h_1}\cdots\psi_{h_t})(\alpha).
	\end{aligned}
	\end{equation}}
	Observe that, in particular, if $H=G$, then $T_{g,k}^H=S_{g,k}$ and $\Q_g^H=\Q_g^G.$
	
	Henceforth, the notation used in this section shall have the same meaning throughout the article unless otherwise stated. 
	\section{Existence of an $r$-primitive $k$-normal pair $(\alpha,\alpha^{-1})$}\label{Sec4}
	In this section, our aim is to show the existence of a pair $(\alpha,\alpha^{-1})$ of $r$-primitive $k$-normal elements in $\mathbb{F}_{q^n}$ over $\mathbb{F}_q.$ From Lemma \ref{L2.1}, it is clear that, if $\gamma\in\mathbb{F}_{q^n}$ is a normal element and $g\in P_k$, then $\alpha=g\circ\gamma$ is a $k$-normal element over $\mathbb{F}_q.$ Therefore, for the existence of a desired pair, we only need to check the $r$-primitivity of $\alpha$ and the $k$-normality of $\alpha^{-1}$. This means for the existence of a desired pair, it is enough to find an $(x^n-1)$-free element $\gamma$, such that $\alpha=g\circ\gamma\in Q_r^R$ and $\alpha^{-1}\in T_{g,k}^G$, for some arbitrary but fixed polynomial $g\in P_k.$ Instead of this, more generally, we will show the existence of an $h$-free element $\gamma$, for some $h\mid x^n-1$, such that $\alpha=g\circ\gamma\in Q_r^d$ and $\alpha^{-1}\in T_{g,k}^H$, where $d\mid R$ and $H\mid G$. Let $N_{r,k,g}(h,d,H)$ denote the number of such elements $\gamma$, then we prove the following result.
	\begin{theorem}\label{T4.1}
		Let $r>0$ and $k\geq 0$ be integers. Further, let $q$ be a prime power and $n$ be a positive integer such that $r\mid q^n-1$ and there exists a polynomial $g\in P_k$.
		Then, $N_{r,k,g}(h,d,H)>0$, if
		\begin{equation*}
			q^{n/2-\vartheta}>2r\hspace{.5mm}\mathrm{rad}(r)W(h)W(d)W(H),
		\end{equation*} 	
		where, $\vartheta=2k$, if $\mathrm{gcd}(q,n)=1$, and $\vartheta=3k$, otherwise.
	\end{theorem}
	\begin{proof}
		Let $g=\sum_{i=0}^{k}a_ix^i\in P_k$ with $a_k=1$, and define $g\circ x=\sum_{i=0}^{k}a_ix^{q^i}$. Further, let $\mathcal{Z}$ be the set of zeroes of the polynomial $g\circ x$. Clearly, $\mathcal{Z}$ consists of all the elements $\beta$, whose $\mathbb{F}_q$-order divides $g$, and hence, $|\mathcal{Z}|=\sum_{g'\mid g}\Phi_q(g')=q^k$. Now, by using the definitions of the characteristic functions $\Upsilon_h$, $\Gamma_r^d$ and $\Q_g^{H}$, we have
		{\small \begin{equation}\label{Eq8}
			\begin{aligned}
			N_{r,k,g}(h,d,H)=&\sum_{\beta\in\mathbb{F}_{q^n}\setminus\mathcal{Z}}\Upsilon_h(\beta)\Gamma_r^d(g\circ\beta)\Q_g^{H}((g\circ\beta)^{-1})\\
			=&\hspace{1mm}\mathcal{H}\sum_{\substack{h'\mid h}}\sum_{\substack{d_1\mid d,d_2\mid u,\\e_j\mid \lambda_j;1\leq j\leq s}}\sum_{\substack{g_1\mid H,g_2\mid\pi,\\h_i,\mid\Lambda_i;1\leq i\leq t}}\frac{\bm{\mu}}{\bm{\Phi}}(h',d_1,g_1)\bm{\ell}\big((e_j)_{j=1}^{s},(h_i)_{i=1}^{t}\big)\sum_{\substack{\psi_{h'}}}\times\\
			&\sum_{\substack{\chi_{d_1},\chi_{d_2},\\\chi_{e_j};1\leq j\leq s}}\sum_{\substack{\psi_{g_1},\psi_{g_2},\\\psi_{h_i};\ 1\leq i\leq t}}\sum_{\beta\in\mathbb{F}_{q^n}\setminus\mathcal{Z}}\psi_{h'}(\beta)\chi_{D}(g\circ\beta)\psi_{F}((g\circ\beta)^{-1}),
			\end{aligned}
			\end{equation}}
		\noindent
		where $\mathcal{H}=\dfrac{\Phi_q(h)\phi(d)\Phi_q(H)}{rdq^{\mathrm{deg}(h)+\mathrm{deg}(H)+\mathrm{deg}(g)}}$, $\dfrac{\bm{\mu}}{\bm{\Phi}}(h',d_1,g_1)=\dfrac{\mu'(h')\mu(d_1)\mu'(g_1)}{\Phi_q(h')\phi(d_1)\Phi_q(g_1)}$, $\bm{\ell}\big((e_j)_{j=1}^{s},(h_i)_{i=1}^{t}\big)=\prod_{j=1}^s\ell_{p_j,e_j}\prod_{i=1}^{t}\ell_{f_i,h_i}'$, $\chi_{D}=\chi_{d_1}\chi_{d_2}\chi_{e_1}\cdots\chi_{e_s}$ is a multiplicative character of order $D=d_1d_2e_1\cdots e_s$, since $d_1,d_2,e_1,\ldots,e_s$ are mutually co-prime, and similarly, $\psi_F=\psi_{g_1}\psi_{g_2}\psi_{h_1}\cdots\psi_{h_t}$ is the additive character of $\mathbb{F}_q$-order $F=g_1g_2h_1\cdots h_t$.
		
		Now, let $\psi_0$ be the canonical additive character of $\mathbb{F}_{q^n}$, then there exist $y_1,y_2\in\mathbb{F}_{q^n}$, such that $\psi_{h'}(\beta)=\psi_{0}(y_1\beta)$ and $\psi_{F}(g\circ\beta)=\psi_{0}(y_2(g\circ\beta))$ for all $\beta\in\mathbb{F}_{q^n}$. Thus,
		{\small
			\begin{align*}
			\sum_{\beta\in\mathbb{F}_{q^n}\setminus\mathcal{Z}}\psi_{h'}(\beta)\chi_{D}(g\circ\beta)\psi_{F}((g\circ\beta)^{-1})=\sum_{\beta\in\mathbb{F}_{q^n}\setminus\mathcal{Z}}\chi_D(g\circ\beta)\psi_0(y_1\beta+y_2(g\circ\beta)^{-1}).
			\end{align*}} 
		Now, we estimate $\big|\sum_{\beta\in\mathbb{F}_{q^n}\setminus\mathcal{Z}}\chi_D(g\circ\beta)\psi_0(y_1\beta+y_2(g\circ\beta)^{-1})\big|$ using Lemma \ref{L2.4}, for which we first show that the rational function $y_1x+y_2(g\circ x)^{-1}$ is not of the form $L(x)^{q^n}-L(x)$ for any rational function $L(x)=\frac{L_1}{L_2}$ (in the simplest form) over the algebraic closure $\mathbb{F}$ of $\mathbb{F}_{q^n}$, unless $y_1=y_2=0$.
		
		Assume that, $y_1x+y_2(g\circ x)^{-1}=(\frac{L_1}{L_2})^{q^n}-\frac{L_1}{L_2}$ for some $L_1, L_2\in \mathbb{F}[x]$ i.e. $L_2^{q^n}(y_1\sum_{i=0}^ka_ix^{q^i+1}+y_2)=\sum_{i=0}^ka_ix^{q^i}(L_1^{q^n}-L_1L_2^{q^n-1}).$ First, let $y_2\neq0$, then $\sum_{i=0}^ka_ix^{q^i}=L_2^{q^n},$ which is not possible because $k<n$. So, $y_2$ must be zero. Additionally, let $y_1\neq0$, then $L_2^{q^n}y_1x=L_1^{q^n}-L_1L_2^{q^n-1}$, which implies $L_2$ is a constant polynomial and $q^n\mid 1$, which is again not possible. Hence, $y_1=y_2=0$. Therefore, if one of the $y_1$ or $y_2$ is non-zero, then by Lemma \ref{L2.4},
		\begin{equation}\tag{I}\label{I}
		\big|\sum_{\beta\in\mathbb{F}_{q^n}\setminus\mathcal{Z}}\chi_D(g\circ\beta)\psi_0(y_1\beta+y_2(g\circ\beta)^{-1})\big|\leq 2q^{n/2+k}.
		\end{equation}
		
		Now, let $y_1=y_2=0$, then $\psi_{h'}$ and $\psi_{F}$ are trivial additive characters, i.e. $h'=1=F$, and the latter one is possible if and only if each factor of $F=g_1g_2h_1\cdots h_t$ is $1$, since $g_1,g_2,h_1,\ldots,h_t$ are mutually co-prime. Hence, we have 
		$$\sum_{\beta\in\mathbb{F}_{q^n}\setminus\mathcal{Z}}\chi_D(g\circ\beta)\psi_0(y_1\beta+y_2(g\circ\beta)^{-1})=\sum_{\beta\in\mathbb{F}_{q^n}\setminus\mathcal{Z}}\chi_D(g\circ\beta).$$
		Notice that, if the polynomial $g\circ x=L(x)^{D}$ for some $L(x)\in \mathbb{F}[x]$, then $D\mid \mathrm{gcd}(q^k,q^n-1)$, i.e $D=1$. Hence, by Lemma \ref{L2.3}, for $D\neq1$, we get
		\begin{equation}\tag{II}\label{II}
		\big|\sum_{\beta\in\mathbb{F}_{q^n}\setminus\mathcal{Z}}\chi_D(g\circ\beta)\big|\leq (q^k-1)q^{n/2}<2q^{n/2+k}.
		\end{equation}

		Finally, let $h'=F=1$ and $D=1$, then
		\begin{equation}\tag{III}\label{III}
		\sum_{\beta\in\mathbb{F}_{q^n}\setminus\mathcal{Z}}\chi_D(g\circ\beta)\psi_0(y_1\beta+y_2(g\circ\beta)^{-1})=q^n-q^k.
		\end{equation} 
		Using the estimates, obtained in \eqref{I}, \eqref{II} and \eqref{III}, in Equation \eqref{Eq8}, we get
		{\small 
			\begin{align*}
			\big|N_{r,k,g}(h,d,H)-&(q^n-q^k)\bm{\ell}\big((1)_{j=1}^{s},(1)_{i=1}^{t}\big)\mathcal{H}\big|\leq\\& \mathcal{H}\underbrace{\sum_{\substack{h'\mid h}}\sum_{\substack{d_1\mid d,d_2\mid u,\\e_j\mid \lambda_j;1\leq j\leq s}}\sum_{\substack{g_1\mid H,g_2\mid\pi,\\h_i,\mid\Lambda_i;1\leq i\leq t}}}_{\text{(all are not 1 simultaneously)}}\big|\frac{\bm{\mu}}{\bm{\Phi}}(h',d_1,g_1)\big|\big|\bm{\ell}\big((e_j)_{j=1}^{s},(h_i)_{i=1}^{t}\big)\big|\times\\
			&\sum_{\substack{\psi_h'}}\sum_{\substack{\chi_{d_1},\chi_{d_2},\\\chi_{e_j};1\leq j\leq s}}\sum_{\substack{\psi_{g_1},\psi_{g_2},\\\psi_{h_i};\ 1\leq i\leq t}}2q^{n/2+k}.
			\end{align*}}
		
		\noindent
		Since $\big|\bm{\ell}\big((e_j)_{j=1}^{s},(h_i)_{i=1}^{t}\big)\big|\leq\bm{\ell}\big((1)_{j=1}^{s},(1)_{i=1}^{t}\big)$, then 
		\begin{align*}
		\big|&N_{r,k,g}(h,d,H)-(q^n-q^k)\bm{\ell}\big((1)_{j=1}^{s},(1)_{i=1}^{t}\big)\mathcal{H}\big|\leq \\&\hspace{.5mm} 2q^{n/2+k}\bm{\ell}\big((1)_{j=1}^{s},(1)_{i=1}^{t}\big)\mathcal{H}\big(u\prod_{j=1}^s\lambda_jq^{\mathrm{deg}(\pi)+\sum_{i=1}^{t}\mathrm{deg}(\Lambda_i)}W(h)W(d)W(H)-1\big).
		\end{align*}
		From the above equation, we get
		\begin{align*}
		\frac{N_{r,k,g}(h,d,H)}{\bm{\ell}\big((1)_{j=1}^{s},(1)_{i=1}^{t}\big)\mathcal{H}}&\geq  (q^n-q^k)-2q^{n/2+k}\times\\&\big(u\prod_{j=1}^s\lambda_jq^{\mathrm{deg}(\pi)+\sum_{i=1}^{t}\mathrm{deg}(\Lambda_i)}
		W(h)W(d)W(H)-1\big).
		\end{align*}
		Clearly, $N_{r,k,g}(h,d,H)>0$, if 
		\begin{equation}\label{Eq9}
		q^{n/2-k}>2u\prod\limits_{j=1}^s\lambda_jq^{\mathrm{deg}(\pi)+\sum_{i=1}^{t}\mathrm{deg}(\Lambda_i)}W(h)W(d)W(H).
		\end{equation}
		Now, if $\mathrm{gcd}(q,n)=1$, then $x^n-1$ has no repeated factor over $\mathbb{F}_q$, and hence, $g=\pi$ and $\mathrm{deg}(\Lambda_i)=0$ for each $i=1,2,\ldots,t.$ Thus,  $q^{\mathrm{deg}(\pi)+\sum_{i=1}^{t}\mathrm{deg}(\Lambda_i)}= q^k$, otherwise $ \leq q^{2k}.$ Moreover, $u\prod\limits_{j=1}^s\lambda_j= r\prod_{j=1}^sp_j\leq r\hspace{.5mm}\mathrm{rad}(r)$. Hence, from Equation \eqref{Eq9}, $N_{r,k,g}(h,d,H)>0$, if $
		q^{n/2-\vartheta}>2r\hspace{.5mm}\mathrm{rad}(r)W(h)W(d)W(H)$, 
		where, $\vartheta=2k$, if $\mathrm{gcd}(q,n)=1$, and $\vartheta=3k$, otherwise.
	\end{proof}
	The following corollary provides the desired existence of $r$-primitive $k$-normal pairs $(\alpha,\alpha^{-1})$ in $\mathbb{F}_{q^n}$ over $\mathbb{F}_q.$
	\begin{cor}\label{cor4.1.1}
		With $r$ and $k$ as in {\upshape Theorem \ref{T4.1}}, there exists a pair $(\alpha,\alpha^{-1})$ of $r$-primitive $k$-normal elements in $\mathbb{F}_{q^n}$ over $\mathbb{F}_q$, i.e. $N_{r,k,g}(x^n-1,R,G)>0$, if
		\begin{equation}\label{Eq10}
		q^{n/2-\vartheta}>2r\hspace{.5mm}\mathrm{rad}(r)W(x^n-1)W(R)W(G),
		\end{equation} 	
		where, $\vartheta=2k$, if $\mathrm{gcd}(q,n)=1$, and $\vartheta=3k$, otherwise.
	\end{cor}
	\begin{proof}
		Let $h=x^n-1$, $d=R$ and $H=G$ in the above theorem, then the result follows immediately.
	\end{proof}
	
	Observe that, Inequalities \eqref{Eq10} can never hold for $n\leq 2\vartheta$, and as the value of $k$ increases, the existence of a pair $(\alpha,\alpha^{-1})$ of $r$-primitive $k$-normal elements in $\mathbb{F}_{q^n}$ over $\mathbb{F}_q$ becomes more challenging for all $n.$ However, we can show that such a pair exists in all but finitely many fields $\mathbb{F}_{q^n}$ over $\mathbb{F}_q$, for any fixed integers $r$ and $k$. For this, we need the following two lemmas.
\begin{lemma}\label{L4.1}
	{\upshape\cite[Lemma 2.9]{Lenstra}}
	Let $q$ be a prime power and $n$ be a positive integer. Then, we have $W(x^n-1)\leq2^{\frac{1}{2} (n+\mathrm{gcd}(n,q-1))}$. In particular, $W(x^n-1)\leq 2^n$ and $W(x^n-1)=2^n$ if and only if $n\mid(q-1)$. Furthermore, $W(x^n-1) \leq 2^{3n/4}$ if $n\nmid(q-1)$, since in this case, $\mathrm{gcd}(n,q-1)\leq n/2.$
\end{lemma}
\begin{lemma}\label{L4.2}{\upshape\cite[Lemma 3.7]{CoHuc}}
	For any $M\in \mathbb{N}$ and a positive real number $\nu$, $W(M)\leq \C M^{1/\nu}$, where $\C=\prod\limits_{i=1}^{t}\frac{2}{p_i^{1/\nu}}$ and $p_1,p_2,\ldots, p_t$ are the primes $\leq 2^\nu $ that divide $M.$
\end{lemma}
\noindent
Following the idea of \cite[Proposition 3.1]{Avnish1}, we have the following result. 
 \begin{prop}\label{P4.1}
 	Let $r\geq1$ and $k\geq0$ be fixed integers. Let $q$ be a prime power and $n>2\vartheta$, where $\vartheta$ is as defined in {\upshape{Theorem \ref{T4.1}}}, such that $r\mid q^n-1$ and there exists a polynomial $g\in P_k$. Then, there exists a pair $(\alpha,\alpha^{-1})$ of  $r$-primitive $k$-normal elements in all but finitely many fields $\mathbb{F}_{q^n}$ over $\mathbb{F}_q.$
 \end{prop}
\begin{proof}
	Notice that, $W(R)\leq W(q^n-1)$ and $W(G)\leq W(x^n-1).$  Hence, from Lemmas \ref{L4.1} and \ref{L4.2}, Inequality \eqref{Eq10} holds if
	\begin{equation}\label{Eq11}
		q^{n/2-\vartheta}>2r\hspace{.5mm}\mathrm{rad}(r)\C q^{n/\nu}2^{2n}.
	\end{equation}
	Set $A=2r\hspace{.5mm}\mathrm{rad}(r)\C$ and $\nu=\mathrm{max}\hspace{.5mm}\{4\vartheta+2,93.46\}.$ Then, taking logarithm on both sides, we get the following equivalent inequality.
	\begin{equation*}
		\mathrm{log}\hspace{.5mm}q>\frac{\mathrm{log}\hspace{.5mm}A+2n\mathrm{log}\hspace{.5mm}2}{n/2-\vartheta-n/\nu}.
	\end{equation*}
	The right hand side of the above inequality is a decreasing function of $n$, if $\frac{n}{2}-\vartheta-\frac{n}{\nu}>0$, i.e. $\frac{\nu-2}{\nu}>\frac{2\vartheta}{n}$, which holds for $n>2\vartheta$. Therefore, there exists a number $q_0$ such that the above inequality holds for $q\geq q_0$ and $n\geq 2\vartheta+1$.
	
	Now, for $q<q_0$, we consider the following inequality, which is equivalent to Inequality \eqref{Eq11}.  
	\begin{equation*}
	n>\frac{\mathrm{log}\hspace{.5mm}A+\vartheta\hspace{.5mm}\mathrm{log}\hspace{.5mm}q}{(1/{2}-1/{\nu})\mathrm{log}\hspace{.5mm}q-2\hspace{.5mm}\mathrm{log}\hspace{.5mm}2}.
	\end{equation*}
	
	Clearly, the denominator of the right hand side of the above inequality is positive, if $q>16$. Hence, corresponding to each $16<q<q_0$, there exists a natural number $n_q$ such that the above inequality holds for $n\geq n_q.$
	
	Now, for $3\leq q\leq 16$, we use the bound $W(x^n-1)<2^{n/3+c_q}$ (from \cite[Lemmas 2.9, 2.11]{Lenstra}), where $c_q$ is a constant corresponding to each $q$. Then, Inequality \eqref{Eq11} holds, if 
	\begin{equation*}
	n>\frac{\mathrm{log}\hspace{.5mm}A+\vartheta\hspace{.5mm}\mathrm{log}\hspace{.5mm}q+2c_q\hspace{.5mm}\mathrm{log}\hspace{.5mm}2}{(1/{2}-1/{\nu})\mathrm{log}\hspace{.5mm}q-(2\hspace{.5mm}\mathrm{log}\hspace{.5mm}2)/{3}}.
	\end{equation*}
	The above inequality holds for $n\geq n_q$, where $n_q$ is a natural number corresponding to each $3\leq q\leq 16$.
	
	Finally, let $q=2$, and from \cite[Lemma 2.11]{Lenstra}, we have $W(x^n-1)<2^{\frac{n-4}{5}}$ for all $n$ except $n\in \{1, 2, 3, 4, 5, 7, 9, 15, 21\}$, and Inequality \eqref{Eq11} holds, if 
	\begin{equation*}
	n>\frac{\mathrm{log}\hspace{.5mm}A+\vartheta\hspace{.5mm}\mathrm{log}\hspace{.5mm}q-(8/5)\hspace{.5mm}\mathrm{log}\hspace{.5mm}2}{(1/{2}-1/{\nu})\mathrm{log}\hspace{.5mm}q-(2\hspace{.5mm}\mathrm{log}\hspace{.5mm}2)/{5}}.
	\end{equation*}
	There exists a natural number $n_2$ such that the last inequality holds for $n\geq n_2.$
	
	Hence, from the above discussion, we conclude that, there exists a pair $(\alpha,\alpha^{-1})$ of $r$-primitive $k$-normal elements in all but finitely many $\mathbb{F}_{q^n}$ over $\mathbb{F}_q.$  
\end{proof}
\subsection{The sieving technique}\label{Sub4.1}
In this section, we give a sieving technique to improve Inequality \eqref{Eq10} by following the methods of Cohen and Huczynska \cite{CoHuc}. The proofs of the next two lemmas are similar to that of \cite[Lemmas 5.1, 5.2]{Rani2}, hence omitted.

\begin{lemma}\label{L4.3}
	Let $r\mid q^n-1$, and let $k$ be a non-negative integer and $g\in P_k$.  
	Further, let $d$ be a divisor of $R$ and $\{p_1,p_2,\ldots, p_l\}$ be the set of remaining distinct primes dividing $R$. Furthermore, let $h$, $H$ be the divisors of $x^n-1$, $G$ respectively, and $\{h_1,h_2,\ldots, h_m\}$, $\{H_1,H_2,\ldots,H_{m'}\}$ be the sets of remaining distinct irreducible factors of $x^n-1$, $G$ respectively. Then 
	\begin{align*}
	N_{r,k,g}(x^n-1,R,G)\geq& \sum_{i=1}^{m}N_{r,k,g}(hh_i,d,H)+\sum_{i=1}^{l}N_{r,k,g}(h,dp_i,H)+\\&\sum_{i=1}^{m'}N_{r,k,g}(h,d,HH_i)-(m+l+m'-1)N_{r,k,g}(h,d,H).
	\end{align*}
\end{lemma}	
\begin{lemma}\label{L4.4}
	With the notations of the above lemma, define $$\D:=1-\sum_{i=1}^{m}\frac{1}{q^{deg(h_i)}}-\sum_{i=1}^{l}\frac{1}{p_i}-\sum_{i=1}^{m'}\frac{1}{q^{deg(H_i)}} \text{ and } \mathcal{S}:=\frac{m+l+m'-1}{\mathcal{D}}+2.$$ Suppose $\mathcal{D}>0,$ then $N_{r,k,g}(x^n-1,R,G)>0$, if
	\begin{equation}\label{Eq12}
	q^{n/2-\vartheta}>2r\hspace{0.5mm}\mathrm{rad}(r)W(h)W(d)W(H)\mathcal{S},
	\end{equation}
	where, $\vartheta=2k$, if $\mathrm{gcd}(q,n)=1$, and $\vartheta=3k$, otherwise.
\end{lemma}
\section{Existence of a primitive $1$-normal pair $(\alpha,\alpha^{-1})$}\label{Sec5}
	In \cite{CoHuc}, Cohen and Huczynska proved the existence of a pair $(\alpha,\alpha^{-1})$ of primitive normal elements in $\mathbb{F}_{q^n}$ over $\mathbb{F}_q$. Here, we present an application of the sufficient condition and the sieving inequality, obtained in the previous section, by showing the existence of a pair $(\alpha,\alpha^{-1})$ of primitive $1$-normal elements in $\mathbb{F}_{q^n}$ over $\mathbb{F}_q$. All the non-trivial calculations wherever needed in this section are done using SageMath \cite{Sage}. 

	From Inequality \eqref{Eq10}, there exist pairs $(\alpha,\alpha^{-1})$ of primitive $1$-normal elements in $\mathbb{F}_{q^n}$ over $\mathbb{F}_q$, if
	\begin{equation}\label{Eq13}
		q^{n/2-\vartheta}>2W(x^n-1)W(R)W(G),
	\end{equation}
	where $\vartheta=2$, if $\mathrm{gcd}(q,n)=1$, otherwise, $\vartheta=3.$ Notice that, if a pair $(q,n)$ satisfies Inequality \eqref{Eq13} in the case of $\vartheta=3$, then it satisfies the same in the case of $\vartheta=2$ as well. Moreover, Inequality \eqref{Eq13} can never hold for $n\leq 4$,  if $\vartheta=2$, and for $n\leq 6$, if $\vartheta=3.$  Therefore, first we assume that $\vartheta=3$ and settle the cases $n\geq 14$ in Subsection \ref{Sub5.1}, and then discuss the cases $5\leq n\leq 13$ in Subsection \ref{Sub5.2}. 
    \subsection{The cases $n\geq14$}\label{Sub5.1}
    From Lemmas \ref{L4.1} and \ref{L4.2}, we have $W(R)\leq W({q^n-1})<\C q^{n/\nu}$, $W(x^n-1)\leq 2^n$ and $W(G)\leq 2^{n-1}$, and hence, Inequality \eqref{Eq13} holds, if 
    \begin{equation}\label{Eq14}
    q^{n/2-3}>\C\hspace{.5mm} q^{n/\nu}2^{2n}.
    \end{equation}First, we study the cases $q>3$ and $n\geq14$ in the following Lemma and then the cases $q=2,3$ and $n\geq14$ in Lemma \ref{L5.3}.
	\begin{lemma}\label{L5.1}
		Let $q>3$ be a prime power and $n\geq 14$ be a positive integer. Then, there always exists a pair $(\alpha,\alpha^{-1})$ of primitive $1$-normal elements in $\mathbb{F}_{q^n}$ over $\mathbb{F}_q$.
\end{lemma}
\begin{proof}
    For $n\geq 14$, if we choose $\nu=7.6$,  Inequality \eqref{Eq14} holds for $q\geq214183$. Now, for $q<214183$,  we rewrite Inequality \eqref{Eq14} as follows
	\begin{equation}\label{Eq15}
	n>\frac{\mathrm{log}\hspace{.5mm}(\C)+3\hspace{.5mm}\mathrm{log}\hspace{.5mm}(q)}{(1/{2}-1/{\nu})\mathrm{log}\hspace{.5mm}q-2\hspace{.5mm}\mathrm{log}\hspace{.5mm}2},
	\end{equation}
	provided $\nu>\frac{2\hspace{.5mm}\mathrm{log}\hspace{.5mm}q}{\mathrm{log}\hspace{.5mm}q-4\hspace{.5mm}\mathrm{log}\hspace{.5mm}2}$ and $q>16.$ Notice that, the value of $\nu$ is large for small values of $q$, and it is difficult to calculate $\C$ in our computer. So, for simplicity, we assume $37\leq q<214183,$ and for each such $q$, we find a natural number $n_q$ (see Table \ref{Table1}) such that Inequality \eqref{Eq15} holds for $n\geq n_q$.
	\begin{table}[t!]
		\caption{Values of $n_q$ for $37\leq q<214183$ using $W(x^n-1)\leq 2^n$.\label{Table1}}
		\centering
		{\scriptsize
			\begin{tabular}{lll|lll|lll}
				$\nu$&$q$&$n_q$&$\nu$&$q$&$n_q$&$\nu$&$q$&$n_q$ \\
				\hline
				7.6&$\geq214183$&14&7.8&307, 311, 313, 317&34&8.3& 107&62\\
				7.5&61747 to 214182&15&7.8&277, 281, 283, 289, 293&35&8.3& 103&64\\
				7.5&24428 to 61746&16&7.8&257, 263, 269, 271&36&8.3& 101&65\\
				7.5& 11926 to 24427&17&7.8&243, 251, 256&37&8.4& 97&67\\
				7.5& 6738 to 11925&18&7.8&229, 233, 239, 241&38&8.5& 89&73\\
				7.5&4231 to 6737&19&7.9&223, 227&39&8.5& 83&79\\
				7.5&2875 to 4230&20&7.9&211&40&8.6& 81&81\\
				7.5&2074 to 2874&21&7.9&197, 199&41&8.6& 79&83\\
				7.5&1569 to 2073&22&7.9&191, 193&42&8.7& 73&92\\
				7.5&1233 to 1568&23&7.9&179, 181&43&8.8& 71&96\\
				7.5&999 to 1232&24&7.9&173&44&8.8&67&104\\
				7.5&829 to 998&25&8&167, 169&45&8.9&64&111\\
				7.5&703 to 828&26&8& 157, 163&47&9& 61&120\\
				7.6&606 to 702&27&8& 149, 151&48&9.1&59&128\\
				7.6&531 to 605&28&8.1& 137, 139&51&9.3& 53&157\\
				7.6&471 to 530&29&8.1& 128, 131&53&9.6& 49&188\\
				7.6& 422 to 470&30&8.1& 125, 127&54&9.7& 47&210\\
				7.6&382 to 421&31&8.1& 121&56&10.1&43&272\\
				7.7&348 to 381&32&8.2& 113&59&10.3& 41&320\\
				7.7&331, 337, 343, 347&33&8.3& 109&61&10.8& 37&489\\
				\hline				
		\end{tabular}}
	\end{table}

	Now, for $3< q<37$, we use the bound $W(x^n-1)\leq 2^{n/3+2(q^2-1)/3}$, which we obtain from \cite[Lemma 2.9]{Lenstra}, and find that, Inequality \eqref{Eq14} holds, if 
\begin{equation}\label{Eq16}
n>\frac{\mathrm{log}\hspace{.5mm}(2^{4(q^2-1)/3}\C)}{(1/{2}-1/{\nu})\mathrm{log}\hspace{.5mm}q-(2\hspace{.5mm}\mathrm{log}\hspace{.5mm}2)/3},
\end{equation}
provided $\nu>\frac{6\hspace{.5mm}\mathrm{log}\hspace{.5mm}q}{3\hspace{.5mm}\mathrm{log}\hspace{.5mm}q-4\hspace{.5mm}\mathrm{log}\hspace{.5mm}2}$ and $q>2$. For these values of $q$, we obtain a natural number $n_q$ (see Table \ref{Table2}) such that Inequality \eqref{Eq16} holds for $n\geq n_q$.

Now, for $n<n_q$ and $q$'s listed in Tables \ref{Table1} and \ref{Table2}, we test Inequality \eqref{Eq13} with $\vartheta=3$, and list the values of $q$ and $n$ in Table \ref{Table3}, which fail to satisfy this inequality.
\begin{table}[h!]
	\caption{Values of $n_q$ for $3<q<37$ using $W(x^n-1)\leq 2^{n/3+2(q^2-1)/3}$.\label{Table2}}
	\centering
	{\scriptsize
		\begin{tabular}{|l|llllllll|}
			\hline
			$\nu$&11.8&11.7&11.6&11.4&11.2&11&10.6&10.3\\
			$q$&32&31&29&27&25&23&19&17\\
			$n_q$&1019&973&883&797&714&636&493&428\\
			\hline	
			\hline
			$\nu$&10.1&9.7&9.4&9&8.8&8.7&8.5&9\\
			$q$&16&13&11&9&8&7&5&4\\
			$n_q$&397&312&263&222&206&195&214&351\\
			\hline
	\end{tabular}}
\end{table}
\begin{table}[t!]
	\caption{ Pairs $(q,n)$ fail to satisfy Inequality \eqref{Eq13} with $\vartheta=3$.\label{Table3}}
	\centering
	{\scriptsize
		\begin{tabular}{p{.5cm}|p{12cm}}
			$n$&$q$\\
			\hline			
			14&4, 5, 8, 9, 11, 13, 23, 25, 27, 29, 41, 43, 64, 71, 113, 125, 127, 169, 197, 211, 239, 281, 337, 379, 421, 449, 463, 491, 547, 617, 631, 659, 673, 701, 729, 743, 883, 911, 953\\
			15&4, 7, 8, 11, 13, 16, 19, 29, 31, 41, 49, 61, 64, 71, 121, 151, 181, 211, 241, 256, 271, 331, 361, 421, 571, 631, 751, 841\\
			16&5, 7, 9, 11, 13, 17, 19, 23, 25, 27, 29, 31, 41, 49, 73, 81, 89, 97, 113, 193, 241, 257, 289, 337, 401\\
			17&4, 16, 103, 256\\
			18&4, 5, 7, 13, 17, 19, 25, 31, 37, 43, 73, 109, 127, 163, 181, 199, 289, 361\\
			20&7, 9, 11, 13, 19, 29, 31, 41, 61, 81, 101, 121\\
			21&4, 8, 13, 16, 43, 64, 169\\
			22&23, 67, 89\\
			23&47\\
			24&5, 7, 11, 13, 17, 19, 25, 37, 49, 73, 97, 121\\
			26&27, 53\\
			28&13, 29\\
			30&4, 7, 11, 19, 31, 61\\
			31&32\\
			32&17\\
			36&5, 19, 37\\
			40&9, 11, 41\\
			42&43\\
			45&4\\
			48&5, 7\\
			63&5\\
			\hline
	\end{tabular}}
\end{table}

Now, for the values of $q$ and $n$, listed in Table \ref{Table3}, we test the SageMath procedure {\small \texttt{TEST_SIEVE(q,n,$\vartheta$)}} with $\vartheta=3$, in which we choose the suitable values of $h$, $d$ and $H$, and test Inequality \eqref{Eq12}  (whose pseudocode is given in \ref{A1}), which returns \texttt{True} for all the pairs $(q,\ n)$ except (4, 15), (5, 16), (5, 24), (8, 14), (9, 16), (16, 15), (17, 16), (19, 18). For these remaining pairs of $(q,\ n)$, we explicitly search for a pair $(\alpha,\alpha^{-1})$ of primitive $1$-normal elements in $\mathbb{F}_{q^n}$ over $\mathbb{F}_q$ using SageMath procedure {\small \texttt{DIRECT_SEARCH(q,n)}} (pseudocode is given in \ref{A1}), and find that such a pair always exists.
\end{proof}
	
	Now, for the cases $q=2,\ 3$, first we recall some results. Let $q$ be a prime power and $n=q^i\cdot n'$ be a positive integer with $\mathrm{gcd}(n',q)=1$; $i\geq0$. Further, let $e$ be the multiplicative order of $q$ modulo $n'$, then $x^{n'}-1$ factors into irreducible polynomials of degree $\leq e.$ Let $\rho(q,n')$ be the ratio of the number of irreducible factors of $x^{n'}-1$ of degree $<e$ to $n'$, then $n\rho(q,n)=n'\rho(q,n')$, and we have the following lemma.
	\begin{lemma}\label{L5.2}{\upshape{\cite[Lemma 7.1]{CoHuc}}}
		Assume that $n>4$ $(p\nmid n)$, then the following holds:
		\begin{enumerate}[$(i)$]
			\item Suppose $q=2$, then  $\rho(2,5)=1/5; \ \rho(2,9)=2/9; \ \rho(2,21)=4/21;$ otherwise $\rho(2,n)\leq1/6.$
			\item Suppose $q=3$, then  $\rho(3,16)=5/16;$ otherwise $\rho(3,n)\leq1/4.$ 
		\end{enumerate}
	\end{lemma} 
If we choose $h$ as the product of all irreducible factors of $x^{n'}-1$ of degree $<e$, $H=h/(x-1)$ and $d=R$ in Lemma \ref{L4.4}, then following Lemma 10 of \cite{Rani1}, we get that, $\mathcal{D}>1-2/e>0$ and $\mathcal{S}\leq 2n'\leq2n$, provided $e>2$ (i.e. $n'\nmid q^2-1$).  Using these results, we prove the following lemma.
\begin{lemma}\label{L5.3}
	Let $q=2, 3$ and $n\geq 14$ be a positive integer. Then, there always exists a pair $(\alpha,\alpha^{-1})$ of primitive $1$-normal elements in $\mathbb{F}_{q^n}$ over $\mathbb{F}_q$. 
\end{lemma}
\begin{proof}
	Let $n=q^i\cdot n'$ with $\mathrm{gcd}(q,n')=1$ and $i\geq 0.$ First, assume that $n'\nmid q^2-1,$ then $e>2$ and $n'>4$. If we choose $d=R$, $h$ as the product of all irreducible factors of $x^{n'}-1$ of degree $<e$, and $H=h/(x-1)$, then from the above discussion and Inequality \eqref{Eq12}, there exists a pair $(\alpha,\alpha^{-1})$ of primitive $1$-normal elements in $\mathbb{F}_{q^n}$ over $\mathbb{F}_q$, if 
	\begin{equation}\label{Eq17}
	q^{n/2-3}>2\C q^{n/\nu}W(h)W(H)\mathcal{S}.
	\end{equation}
	Notice that, $W(H)=W(h)/2$ and $W(h)=2^{n'\rho(q,n')}$. From Lemma \ref{L5.2} $n'\rho(q,n')\leq n'/6\leq n/6,$ when $q=2$, and $n'\rho(q,n')\leq n'/4\leq n/4,$ when $q=3$. Thus, the above inequality holds, if
	\begin{equation*}
	\left\{\begin{array}{ll}
	q^{n/2-3}>\C q^{n/\nu} 2^{n/3}2n,& \text{ for } q=2,\\
	q^{n/2-3}>\C q^{n/\nu} 2^{n/2}2n,& \text{ for } q=3.
	\end{array}\right. 
	\end{equation*}
	The above inequalities hold for $n\geq557$, when $q=2$ and $\nu=8.7$, and for $n\geq 265$, when $q=3$ and $\nu=8.3.$  For the remaining $q$ and $n$, we test Inequality \eqref{Eq13}, and find that this fails for $n=$ 14, 15, 17, 18, 20, 21, 22, 27, 28, 30, 31, 33, 35, 42, 45, 63, when $q=2$, and for $n=$ 14, 15, 16, 20, 22, 26, 28, 32, 40, when $q=3.$
	
	Now, assume that $n'\mid q^2-1$, and select $d=R$, $h$ as the product of all the linear factors of $x^{n'}-1$, and $H=h/(x-1)$ in Lemma \ref{L4.4}. Then, for $q=2$, there exists exactly one linear and at most 1 quadratic factor of $x^{n'}-1$ over $\mathbb{F}_2$. Therefore, in this case,  $\mathcal{D}\geq1/2$ and $\mathcal{S}\leq 4$. Similarly, for $q=3$, there exists at most 2 linear and at most 2 quadratic factors of $x^{n'}-1$ over $\mathbb{F}_3$. Thus, in this case, $\mathcal{D}\geq1/3$ and $\mathcal{S}\leq 17$. If we take $\nu=6$, then Inequality \eqref{Eq17} holds for $n\geq 33$ in both the cases. For $14\leq n<33$, we test Inequality \eqref{Eq13}, which does not hold for $n=$ 16, when $q=2$, and for $n=$ 18, 24, when $q=3.$
	
	Finally, for $q=2, 3$ and the remaining values of $n$, we apply the SageMath procedure {\small \texttt{TEST_SIEVE(q,n,$\vartheta$)}} with $\vartheta=3$, which returns \texttt{False} only for $q=2$ and $n=$ 14, 15, 16, 18, 20, 21, 24, 30, and for $q=3$ and $n=$ 14, 16. For these remaining $q$ and $n$, we explicitly search for a pair $(\alpha,\alpha^{-1})$ of primitive $1$-normal elements in $\mathbb{F}_{q^n}$ over $\mathbb{F}_q$ using SageMath procedure {\small \texttt{DIRECT_SEARCH(q,n)}}, and find that such a pair always exists.
\end{proof}
\subsection{The cases $5\leq n\leq13$}\label{Sub5.2}
	Observe that, Inequality \eqref{Eq13} is applicable for the cases, $n=5$, 6, only if $\mathrm{gcd}(n,q)=1$. Therefore, first we settle the cases $7\leq n\leq 13$ for all $q\geq 2$, and then the cases $n=5$, 6 with the assumption that $\mathrm{gcd}(q,n)=1.$ 
	Before moving ahead, first we state the following lemma, which is a direct consequence of Lemma \ref{L4.4}.
\begin{lemma}\label{L5.4}
	Let $d$ be a divisor of $q^n-1$ and $n_0$ be a fixed positive integer such that all the prime divisors of $(q^n-1)/d$, co-prime to $d$, are of the form $n_0s+1$ for some positive integer $s$. Further, let $p_1, p_2,\ldots,p_l$ be the first $l$ primes of the form $n_0s+1$ such that the product $P_l=\prod_{i=1}^{l}p_i\leq(q^n-1)/d$, and denote the sum of their reciprocals by $S_l=\sum_{i=1}^{l}{1}/{p_i}$. If  $h=1$ and $H=1$ are the divisors of $x^n-1$ and $(x^n-1)/(x-1)$, respectively,  then $\D\geq1-S_l-(2n-1)/q$ and $\mathcal{S}\leq(l+2n-2)/\D+2.$ Moreover, if $\D>0$ and $q^{n/2-\vartheta}>2W(d)\mathcal{S}$, then there exists a pair $(\alpha,\alpha^{-1})$ of primitive $1$-normal elements in $\mathbb{F}_{q^n}$ over $\mathbb{F}_q$, where $\D$, $\mathcal{S}$ and $\vartheta$ are the same as defined in {\upshape{Lemma \ref{L4.4}}} and {\upshape{Inequality \eqref{Eq13}}}.
\end{lemma}
\begin{lemma}\label{L5.5}
	Let $7\leq n\leq 13$ and $q$ be a prime power such that $\mathrm{gcd}(q,n)=1$. Then, there always exists a pair $(\alpha,\alpha^{-1})$ of primitive $1$-normal elements in $\mathbb{F}_{q^n}$ over $\mathbb{F}_q$.
\end{lemma}
\begin{proof}
	Since, $\mathrm{gcd}(q,n)=1$, Inequality \eqref{Eq13} holds, if 
	\begin{equation}\label{Eq18}
		q^{n/2-2}>\C q^{n/\nu}2^{2n}.	
	\end{equation} 
	For each $7\leq n\leq 13$ and suitable choices of $\nu$, the above inequality holds for $q\geq M_n$, where the values of $M_n$ are listed in Table \ref{Table4}. 
	\begin{table}[h!]
		\caption{Values of $M_n$ such that Inequality \eqref{Eq18} holds for $q\geq M_n$.}\label{Table4}
		\centering
		{\scriptsize
			\begin{tabular}{|l|lllllll|}
				\hline
				$M_n$&$3.60\times 10^{12}$&$2.10\times 10^{8}$&$2.60\times 10^{6}$&$222280$&$46892$&$16049$&$7331$\\
				\hline
				$n$&7&8&9&10&11&12&13\\
				\hline
				$\nu$&7.8&7.4&7.2&7.1&7.1&7.1&7.1\\
				\hline
		\end{tabular}}
	\end{table} 
Now, for each $10\leq n\leq 13,$ and $q<M_n$, by calculating the exact values of $W(x^n-1)$, $W(R)$ and $W(G)$, and testing Inequality \eqref{Eq13} with $\vartheta=2$, we obtain that, there are $136$ pairs of $(q,n)$, which fail to satisfy this inequality. For these pairs, we use the SageMath procedure \texttt{ TEST_SIEVE(q,n,$\vartheta$)} with $\vartheta=2$, which returns \texttt{False} for the pairs $(3,10)$, $(11,10)$, $(2,11)$, $(3,11)$, $(5,12)$, $(7,12)$, and $(13,12)$. Finally, for these pairs,  we explicitly search for a pair $(\alpha,\alpha^{-1})$ of primitive $1$-normal elements in $\mathbb{F}_{q^n}$ over $\mathbb{F}_q$ using SageMath procedure {\small \texttt{DIRECT_SEARCH(q,n)}}, and find that such a pair always exists.

Now, let $n=7$, then for $q<M_7$, we use Lemma \ref{L5.4} with $d=q-1$, $h=1$, $H=1$ and $\vartheta=2.$ In this case, the primes $p$ dividing $(q^7-1)/d$, co-prime to $d$, are of the form $7s+1$, and $P_l\leq \sum_{i=0}^{6}M_7^{i}$ for $l\leq 29$. Further, if $q\geq 65$, then  $\D>0.6608$ and $\mathcal{S}\leq 64.042$, and there exists a desired pair $(\alpha,\alpha^{-1})$, if $q^{3/2}>2\C q^{1/\nu}\mathcal{S}$, which is true for $q\geq 149$ and $\nu=2.8.$ Proceeding in this way for $65\leq q<149$, we get that, the inequality $q^{3/2}>2\C q^{1/\nu}\mathcal{S}$ holds for all $68\leq q\leq 149.$ Now, for $q<68$, we use the SageMath procedure {\small \texttt{TEST_SIEVE(q,n,$\vartheta$)}} with $\vartheta=2$, which returns \texttt{False} for $q=$ 2, 3, 4, 5, 8, 9, 11, 13, and $n=$ 7. In the end, for these pairs,  we explicitly search for a pair $(\alpha,\alpha^{-1})$ of primitive $1$-normal elements in $\mathbb{F}_{q^n}$ over $\mathbb{F}_q$ using SageMath procedure {\small \texttt{DIRECT_SEARCH(q,n)}}, and find that such a pair always exists.

Finally, for the cases $n=8$, 9, and $q<M_n$, again we use Lemma \ref{L5.4} with $d=q^4-1$, $q^3-1$, respectively. In these cases, the primes $p$ dividing $(q^n-1)/d$, co-prime to $d$, are of the form $8s+1$, when $n=8$, and of the form $9s+1$, when $n=9$. Now, following the same line as in the case $n=7$, we have reduced the bounds on $q$ from $M_8$ to 203, and from $M_9$ to 41.Now, for $q<68$, we use the SageMath procedure {\small \texttt{TEST_SIEVE(q,n,$\vartheta$)}} with $\vartheta=2$, which returns \texttt{False} for $q=$ 3, 5, 7, 9, 11, 13, 17 and $n=$ 8, and $q=$ 2, 4, 7, 19, and $n=9$. Finally, for these pairs,  we explicitly search for a pair $(\alpha,\alpha^{-1})$ of primitive $1$-normal elements in $\mathbb{F}_{q^n}$ over $\mathbb{F}_q$ using SageMath procedure {\small \texttt{DIRECT_SEARCH(q,n)}}, and find that such a pair always exists.
\end{proof}
\begin{lemma}\label{L5.6}
	Let $7\leq n\leq 13$ and $q$ be a prime power such that $\mathrm{gcd}(q,n)\neq1$. Then, there always exists a pair $(\alpha,\alpha^{-1})$ of primitive $1$-normal elements in $\mathbb{F}_{q^n}$ over $\mathbb{F}_q$.
\end{lemma}
\begin{proof}
	Since, $\mathrm{gcd}(q,n)\neq1$, $q=p^t$, for each prime $p$ dividing $n$ and some positive integer $t$, and Inequality \eqref{Eq13} holds, if 
	\begin{equation}\label{Eq19}
		p^{nt/2-3}>\C p^{nt/\nu}2^{2n}
	\end{equation}
	For each $7\leq n\leq13$, the above inequality holds for $t\geq t_{n,p}$ and the suitable choices of $\nu$, listed in Table \ref{Table5}. 
	\begin{table}[h!]
		\caption{Values of $t_{n,p}$ such that Inequality \eqref{Eq19} holds for $t\geq t_{n,p}$.}\label{Table5}
		\centering
		{\scriptsize
			\begin{tabular}{|l|lllllll|}
				\hline
				$t_{n,p}$&$3797$&$223$&$46$&$42,19$&$9$&$24,15$&$6$\\
				\hline
				$(n,p)$&$(7,7)$&$(8,2)$&$(9,3)$&$(10,2), (10,5)$&$(11,11)$&$(12,2), (12,3)$&$(13,13)$\\
				\hline
				$\nu$&15.7&10.3&9&8&7.5&7.5&7\\
				\hline
		\end{tabular}}
	\end{table} 
For the cases $9\leq n\leq 13$, and $t<t_{n,p}$, by calculating the exact values of $W(x^n-1)$, $W(R)$ and $W(G)$, and testing Inequality \eqref{Eq13} with $\vartheta=3$, we obtain that, there are $25$ pairs of $(q,n)$, which fail to satisfy this inequality. For these pairs, we use the SageMath procedure {\small \texttt{TEST_SIEVE(q,n,$\vartheta$)}} with $\vartheta=3$, which returns \texttt{False} for the pairs $(3,9)$, $(9,9)$, $(2,10)$, $(4,10)$, $(16,10)$, $(5,10)$, $(2,12)$, $(4,12)$, $(3,12)$ and $(9,12)$. Finally, for these pairs,  we explicitly search for a pair $(\alpha,\alpha^{-1})$ of primitive $1$-normal elements in $\mathbb{F}_{q^n}$ over $\mathbb{F}_q$ using SageMath procedure {\small \texttt{DIRECT_SEARCH(q,n)}}, and find that such a pair always exists.

Now, let $n=7$, $q=7^t$ and $t<t_{7,7}$. Clearly, $x^7-1$ has only one linear factor over $\mathbb{F}_q$ and $7\nmid q^7-1$, therefore, using Lemma \ref{L4.4} with $d=q-1$, $h=1$ and $H=1$, we get that there exists a pair $(\alpha, \alpha^{-1})$ of desired properties, if $7^{t/2}>2\C 7^{t/\nu}\mathcal{S}$. Since, the primes dividing $(q^7-1)/d$ are of the form $7s+1$, so let $P_l$ and $S_l$ be the product and sum of reciprocals of first $l$ primes of the form $7s+1$ respectively, such that $P_l\leq\sum_{i=0}^{6}7^{it}$, then $l\leq 3891$, and from Lemma \ref{L4.4}, $\D\geq 1-S_l-1/7^t$ and $\mathcal{S}\leq l/D+2.$ Now, For $t\geq8$, we get $\D>0.7661$ and $\mathcal{S}\leq 5080.82$, and inequality 
$7^{t/2}>2\C 7^{t/\nu}\mathcal{S}$ holds for $t\geq 20$ and $\nu=5$. Proceeding in this way for $8\leq t< 20$, we get that, the inequality $7^{t/2}>2\C 7^{t/\nu}\mathcal{S}$ always holds for $t\geq 11.$ In the end, for $1\leq t\leq10$, we first perform the SageMath procedure {\small \texttt{TEST_SIEVE(q,n,$\vartheta$)}} with $\vartheta=3$, which returns \texttt{False} only for $t=$ 1, 2, 3, 4, and then for these values,  we explicitly search for a pair $(\alpha,\alpha^{-1})$ of primitive $1$-normal elements in $\mathbb{F}_{q^n}$ over $\mathbb{F}_q$ using SageMath procedure {\small \texttt{DIRECT_SEARCH(q,n)}}, and find that such a pair always exists.

Now, in the case $n=8$, we have $q=2^t$ and $x^8-1=(x-1)^8$ over $\mathbb{F}_q.$ If $d=q^4-1$ and $p$ is a prime divisor of $q^4+1$ co-prime to $d$, then $q^4\not\equiv1(\mathrm{mod}\hspace{0.5mm}p)$
and $q^4\equiv-1(\mathrm{mod}\hspace{0.5mm}p)$, which implies that, the multiplicative order of $q$ modulo $p$ is $8$. Thus, $p$ is of the form $8s+1.$ Again, we use Lemma \ref{L4.4} with $d=q^4-1$, $h=H=1$. Let $P_l$ and $S_l$ be the product and sum of reciprocals of first $l$ primes of the form $8s+1$ respectively, such that $P_l\leq2^{4t}+1$, then $l\leq 90$ for $t<t_{8,2}$, and from Lemma \ref{L4.4}, $\D\geq 1-S_l-1/2^t$ and $\mathcal{S}\leq l/D+2.$ Now, For $t\geq8$, we get $\D>0.764$ and $\mathcal{S}\leq 119.75$, and inequality 
$2^t>2\C 2^{4t/\nu}\mathcal{S}$ holds for $t\geq 35$ and $\nu=7$. Proceeding in this way for $8\leq t< 35$, we get that, the inequality $2^t>2\C q^{4t/\nu}\mathcal{S}$ always holds for $t\geq 30.$ In the end, for $1\leq t\leq29$, we obtain that, there are 14 values of $t$ ($1\leq t\leq 12;\ t= 15$, 18) which fails to satisfy Inequality \eqref{Eq13}. For these values, we use the SageMath procedure {\small \texttt{TEST_SIEVE(q,n,$\vartheta$)}} with $\vartheta=3$, which returns \texttt{False} for $1\leq t\leq 6$. Finally, we explicitly search for a pair $(\alpha,\alpha^{-1})$ of primitive $1$-normal elements in $\mathbb{F}_{q^n}$ over $\mathbb{F}_q$ using SageMath procedure {\small \texttt{DIRECT_SEARCH(q,n)}}, and find that such a pair always exists.
\end{proof}

\begin{lemma}\label{L5.7}
	Let $n=$ $5$, $6$, and $q$ be a prime power such that $\mathrm{gcd}(q,n)=1$. Then, there always exists a pair $(\alpha,\alpha^{-1})$ of primitive $1$-normal elements in $\mathbb{F}_{q^n}$ over $\mathbb{F}_q$ with the sole genuine exception $(4,5)$. 
\end{lemma}
\begin{proof}
	First, we test Inequality \eqref{Eq18} in both the cases $n=$ 5, 6, and get that, this holds for $q\geq M_n$, where $M_5=6.97\times 10^{323}$, if we set $\nu=11.9$, and $M_6=4.74\times 10^{27}$, if we set $\nu=8.7.$ Now, for $q< M_n$, in each case, we use Lemma \ref{L5.4} with $d=q-1$ for $n=5$, and $d=q^2-1$ for $n=6.$ Notice that, primes dividing $(q^n-1)/d$, co-prime to $d$, are of the form $5s+1$ for $n=5$, and $3s+1$ for $n=6.$ Then, for $n=5$ and $q\geq10^5$, we get $l\leq 360$, $\D>0.648$, $\mathcal{S}\leq 569.9$, and  for $n=6$ and $q\geq10^4$, we get $l\leq 48$, $\D>0.429$, $\mathcal{S}\leq 137.11.$ With these, we test the inequalities $q^{1/2}>2\C q^{1/\nu}\mathcal{S}$ for $n=5; \ \nu=5.5$, and $q>2\C q^{2/\nu}\mathcal{S}$ for $n=6;\ \nu=5.2$, and get that, there always exists a pair $(\alpha,\alpha^{-1})$ of desired properties for $q\geq 4.395\times 10^{13}$ and $n=5$, and $q\geq 646882$ and $n=6.$ By repeating the process for remaining values of $q$ in each case, we significantly reduce the bound on $q$ from $M_5$ to $4.96\times 10^9$, and from $M_6$ to $94531.$
	
	Now, for $n=5;$ $6$ and $q<4.96\times 10^9;$ $94531$, respectively, we use Lemma \ref{L5.4} with $d=\mathrm{gcd}(2\cdot3\cdot5,q^n-1)$. Notice that, the primes dividing $(q^n-1)/d$ co-prime to $d$ are of the form $2s+1$ in each case. Then, for $10^5\leq q<4.96\times 10^9$ and $n=5$, we get $l\leq 28$, $\D>0.175$, $\mathcal{S}\leq 207.05$, and the inequality $q^{1/2}>2^4\mathcal{S}$ holds for $q\geq10973928$. Similarly, for $10^3\leq q <94531$ and $n=6$, we get $l\leq 19$, $\D>0.253$, $\mathcal{S}\leq 116.58$, and the inequality $q>2^4\mathcal{S}$ holds for $q\geq1866$. Proceeding in this way, we get that, there always exists a pair $(\alpha, \alpha^{-1})$ of desired properties for $q\geq 3839356$ and $n=5$, and $q\geq1109$ and $n=6$. For the remaining values of $q$ in each case, we run the SageMath procedure {\small \texttt{TEST_SIEVE(q,n,$\vartheta$)}} with $\vartheta=2$ and get that,  there are $871$ values of $q$  (listed in Table \ref{Table6}) in the case of $n=5$, and $47$ values of $q$  (listed in Table \ref{Table6}) in the case of $n=6$, for which this procedure returns \texttt{False}. Finally, we perform the procedure {\small \texttt{DIRECT_SEARCH(q,n)}} for the remaining values of $t$, and find that, there always exists the desired pairs $(\alpha, \alpha^{-1})$ with the sole genuine exception $(4,5).$ 
\end{proof}
	Summing up Lemmas \ref{L5.1}, \ref{L5.3}, \ref{L5.5}, \ref{L5.6}, and \ref{L5.7}, we conclude our Theorem \ref{T1.3}.
\subsection{A note on the cases $n=1,2,3,4$}
	
	From the definition of $k$-normal elements in $\mathbb{F}_{q^n}$ over $\mathbb{F}_q$, it is clear that $0\leq k\leq n-1$, therefore, it is sensible to talk about the existence of a pair $(\alpha,\alpha^{-1})$ of primitive $1$-normal elements in $\mathbb{F}_{q^n}$ over $\mathbb{F}_q$ for $n>1$. Now, for the case $n=2$, if $\alpha$ is a $1$-normal element in $\mathbb{F}_{q^n}$ over $\mathbb{F}_q$, then either $\alpha^q-\alpha=0$ or $\alpha^q+\alpha=0$, which implies $\mathrm{ord}(\alpha)\leq 2(q-1)<q^2-1$. Thus, there does not exist any pair $(\alpha, \alpha^{-1})$ of primitive $1$-normal elements in $\mathbb{F}_{q^2}$ over $\mathbb{F}_q.$ Further, we show that, such a pair does not exists in the case of $n=3$ also.  
	
	\begin{lemma}\label{L5.8}
		Let $n=3$ and $q\geq 2$ be any prime power. Then, there does not exist any pair $(\alpha,\alpha^{-1})$ of primitive $1$-normal elements in $\mathbb{F}_{q^n}$ over $\mathbb{F}_q.$
	\end{lemma}
	\begin{proof}
		Suppose that, there exists a pair $(\alpha,\alpha^{-1})$ of primitive $1$-normal elements in $\mathbb{F}_{q^3}$ over $\mathbb{F}_q$.
		Since, $x^3-1=(x-1)^3$ or $(x-1)(x^2+x+1)$ or $(x-1)(x-a)(x-a^{-1})$ for some $a\in\mathbb{F}_q\setminus\{0,1\}$. Therefore, we have the following cases on the $\mathbb{F}_q$-orders of $\alpha$ and $\alpha^{-1}$.
		
		\noindent
		{\bf Case 1.} If $\mathrm{Ord}_q(\alpha)=\mathrm{Ord}_q(\alpha^{-1})=x^2+x+1$, then, the trace of $\alpha$ and $\alpha^{-1}$ over $\mathbb{F}_q$ is equal to zero. Consequently, the primitive polynomial of $\alpha$ will be of the form $x^3+a$ for some $a\in \mathbb{F}_q,$ and hence $\alpha^3\in \mathbb{F}_q.$ This means $q^3-1\mid 3(q-1)$, which is never possible for any $q>1$.
		
		\noindent
		{\bf Case 2.} If $(x-1)\mid\mathrm{gcd}(\mathrm{Ord}_q(\alpha),\mathrm{Ord}_q(\alpha^{-1}))$, then we get that, $(\alpha^{q}-\alpha)^{q-1}$ and $(\alpha^{-q}-\alpha^{-1})^{q-1}$ both belong to $\mathbb{F}_{q}.$ This implies $\alpha^{q^2-1}\in \mathbb{F}_q$, i.e. $q^3-1\mid (q^2-1)(q-1)$, which is not possible.
		
		\noindent
		{\bf Case 3.} If $\mathrm{Ord}_q(\alpha)=(x-1)(x-a)$ and $\mathrm{Ord}_q(\alpha^{-1})=(x-a)(x-a^{-1})$, then we get that, $(\alpha^{q}-a\alpha)^{q-1}=1$ and $(\alpha^{-q}-a^{-1}\alpha^{-1})^{q-1}\in\mathbb{F}_{q}.$ This implies $\alpha^{q^2-1}\in \mathbb{F}_q$, i.e. $q^3-1\mid (q^2-1)(q-1)$, which is not possible.
		
		\noindent
		{\bf Case 4.} If $\mathrm{Ord}_q(\alpha)=(x-1)(x-a^{-1})$ and $\mathrm{Ord}_q(\alpha^{-1})=(x-a)(x-a)$,then by interchanging the role of $\alpha$ and $\alpha^{-1}$ in Case 3, we get that, $q^3-1\mid (q^2-1)(q-1)$, which is not possible.
		
		This completes the proof of the lemma.
	\end{proof}
 We also provide a necessary condition for the existence of a pair $(\alpha,\alpha^{-1})$ of primitive $1$-normal elements in $\mathbb{F}_{q^4}$ over $\mathbb{F}_q.$
\begin{lemma}\label{L5.9}
	If there exists a pair $(\alpha,\alpha^{-1})$ of primitive $1$-normal elements in $\mathbb{F}_{q^4}$ over $\mathbb{F}_q$, then $q\equiv1\hspace{0.5mm}(\mathrm{mod}\hspace{0.5mm4}).$
\end{lemma}
\begin{proof}
	Let's assume that $q\not\equiv1\hspace{0.5mm}(\mathrm{mod}\hspace{0.5mm4})$, then the polynomial $x^4-1$ factorizes  either as $(x-1)^4$, or as $(x-1)(x+1)(x^2+1)$ over $\mathbb{F}_q.$ Further, let $(\alpha,\alpha^{-1})$ be a pair of $1$-normal elements in $\mathbb{F}_{q^4}$ over $\mathbb{F}_q$. Then, we have the following cases.

	\noindent
	{\bf Case 1:} $(x-1)\mid \mathrm{gcd}(\mathrm{Ord}_q(\alpha),\mathrm{Ord}_q(\alpha^{-1})),$ then $\alpha^{q^2}-\alpha$ and $\alpha^{-q^2}-\alpha^{-1}$ both will belong to $\mathbb{F}_q$. This implies $\alpha^{q^2+1}\in \mathbb{F}_q$, i.e. $\mathrm{ord}(\alpha)\leq (q^2+1)(q-1)<q^4-1.$
	
	\noindent
	{\bf Case 2:} $(x+1)\mid \mathrm{gcd}(\mathrm{Ord}_q(\alpha),\mathrm{Ord}_q(\alpha^{-1})),$ then $(\alpha^{q^2}-\alpha)^2$ and $(\alpha^{-q^2}-\alpha^{-1})^2$ both will belong to $\mathbb{F}_q$. This implies $\alpha^{2(q^2+1)}\in \mathbb{F}_q$, i.e. $\mathrm{ord}(\alpha)\leq 2(q^2+1)(q-1)<q^4-1.$
	
	\noindent
	{\bf Case 3:} $ \mathrm{gcd}(\mathrm{Ord}_q(\alpha),\mathrm{Ord}_q(\alpha^{-1}))=x^2+1$, then either $\alpha^{q^2}-\alpha$ and $(\alpha^{-q^2}-\alpha^{-1})^2$ belong to $\mathbb{F}_q$, or $(\alpha^{q^2}-\alpha)^2$ and $\alpha^{-q^2}-\alpha^{-1}$ will belong to $\mathbb{F}_q$ simultaneously. In either case, we get that $\mathrm{ord}(\alpha)\leq 2(q^2+1)(q-1)<q^4-1.$
	
	In each case, we get that $\mathrm{ord}(\alpha)<q^4-1$, which implies $\alpha$ is never a primitive element in $\mathbb{F}_q.$ Thus, a pair $(\alpha,\alpha^{-1})$ of primitive $1$-normal elements in $\mathbb{F}_{q^4}$ over $\mathbb{F}_q$ can exist, only if $q\equiv1\hspace{0.5mm}(\mathrm{mod}\hspace{0.5mm}4).$
\end{proof}
For the cases, $n=4$ and $q=4t+1$; $n=5$ and $q=5^t$; and $n=6$ and $q=2^t, 3^t$, where $t$ is a positive integer, we performed the extensive experiments, and obtained that, there always exists a pair $(\alpha,\alpha^{-1})$ of primitive $1$-normal in $\mathbb{F}_{q^n}$ over $\mathbb{F}_q$ unless $n=6$ and $q=2,4$. On the basis of these experiments, we conjecture the following.
\begin{con}
	Let $t$ be a positive integer. There always exists a pair $(\alpha,\alpha^{-1})$ of primitive $1$-normal in $\mathbb{F}_{q^n}$ over $\mathbb{F}_q$ if $n=4$ and $q=4t+1;$ $n=5$ and $q=5^t;$ and $n=6$ and $q=2^t, 3^t$ with the sole genuine exceptions $(2,6)$, $(4,6)$. 
\end{con} 
\section{Acknowledgements}
This work was funded by Council of Scientific $\&$ Industrial Research, under Grant F. No. 09/045(1674)/2019-EMR-I and University Grant Commission, under Grant Ref. No. 1042/CSIR-UGC NET DEC-2018.

\bibliographystyle{elsarticle-harv}
\bibliography{inverse_of_r_primitive_k_norm.bib}

\begin{thebibliography}{20}
\expandafter\ifx\csname natexlab\endcsname\relax\def\natexlab#1{#1}\fi
\providecommand{\url}[1]{\texttt{#1}}
\providecommand{\href}[2]{#2}
\providecommand{\path}[1]{#1}
\providecommand{\DOIprefix}{doi:}
\providecommand{\ArXivprefix}{arXiv:}
\providecommand{\URLprefix}{URL: }
\providecommand{\Pubmedprefix}{pmid:}
\providecommand{\doi}[1]{\href{http://dx.doi.org/#1}{\path{#1}}}
\providecommand{\Pubmed}[1]{\href{pmid:#1}{\path{#1}}}
\providecommand{\bibinfo}[2]{#2}
\ifx\xfnm\relax \def\xfnm[#1]{\unskip,\space#1}\fi
\bibitem[{Aguirre and Neumann(2021)}]{newman}
\bibinfo{author}{Aguirre, J.J.}, \bibinfo{author}{Neumann, V.G.},
  \bibinfo{year}{2021}.
\newblock \bibinfo{title}{Existence of primitive 2-normal elements in finite
  fields}.
\newblock \bibinfo{journal}{Finite Fields and Their Applications}
  \bibinfo{volume}{73}, \bibinfo{pages}{101864}.
\bibitem[{Blum and Micali(1984)}]{blum}
\bibinfo{author}{Blum, M.}, \bibinfo{author}{Micali, S.}, \bibinfo{year}{1984}.
\newblock \bibinfo{title}{How to generate cryptographically strong sequences of
  pseudorandom bits}.
\newblock \bibinfo{journal}{SIAM journal on Computing} \bibinfo{volume}{13},
  \bibinfo{pages}{850--864}.
\bibitem[{Cohen and Huczynska(2010)}]{CoHuc}
\bibinfo{author}{Cohen, S.D.}, \bibinfo{author}{Huczynska, S.},
  \bibinfo{year}{2010}.
\newblock \bibinfo{title}{The strong primitive normal basis theorem}.
\newblock \bibinfo{journal}{Acta Arith.} \bibinfo{volume}{143},
  \bibinfo{pages}{299--332}.
\newblock \URLprefix \url{https://doi.org/10.4064/aa143-4-1},
  \DOIprefix\doi{10.4064/aa143-4-1}.
\bibitem[{Cohen and Kapetanakis(2021)}]{cohenrprim}
\bibinfo{author}{Cohen, S.D.}, \bibinfo{author}{Kapetanakis, G.},
  \bibinfo{year}{2021}.
\newblock \bibinfo{title}{Finite field extensions with the line or translate
  property for-primitive elements}.
\newblock \bibinfo{journal}{Journal of the Australian Mathematical Society}
  \bibinfo{volume}{111}, \bibinfo{pages}{313--319}.
\bibitem[{Diffie and Hellman(1976)}]{diffie}
\bibinfo{author}{Diffie, W.}, \bibinfo{author}{Hellman, M.},
  \bibinfo{year}{1976}.
\newblock \bibinfo{title}{New directions in cryptography}.
\newblock \bibinfo{journal}{IEEE transactions on Information Theory}
  \bibinfo{volume}{22}, \bibinfo{pages}{644--654}.
\bibitem[{{Fu} and {Wan}(2014)}]{Fu}
\bibinfo{author}{{Fu}, L.}, \bibinfo{author}{{Wan}, D.}, \bibinfo{year}{2014}.
\newblock \bibinfo{title}{A class of incomplete character sums}.
\newblock \bibinfo{journal}{Quarterly Journal of Mathematics}
  \bibinfo{volume}{65}, \bibinfo{pages}{1195--1211}.
\bibitem[{Gao(1999)}]{gao}
\bibinfo{author}{Gao, S.}, \bibinfo{year}{1999}.
\newblock \bibinfo{title}{Elements of provable high orders in finite fields}.
\newblock \bibinfo{journal}{Proceedings of the American Mathematical Society}
  \bibinfo{volume}{127}, \bibinfo{pages}{1615--1623}.
\bibitem[{Huczynska et~al.(2013)Huczynska, Mullen, Panario and Thomson}]{huc13}
\bibinfo{author}{Huczynska, S.}, \bibinfo{author}{Mullen, G.L.},
  \bibinfo{author}{Panario, D.}, \bibinfo{author}{Thomson, D.},
  \bibinfo{year}{2013}.
\newblock \bibinfo{title}{Existence and properties of k-normal elements over
  finite fields}.
\newblock \bibinfo{journal}{Finite Fields and Their Applications}
  \bibinfo{volume}{24}, \bibinfo{pages}{170--183}.
\bibitem[{Kapetanakis and Reis(2019)}]{kapevar}
\bibinfo{author}{Kapetanakis, G.}, \bibinfo{author}{Reis, L.},
  \bibinfo{year}{2019}.
\newblock \bibinfo{title}{Variations of the primitive normal basis theorem}.
\newblock \bibinfo{journal}{Designs, Codes and Cryptography}
  \bibinfo{volume}{87}, \bibinfo{pages}{1459--1480}.
\bibitem[{Lenstra and Schoof(1987)}]{Lenstra}
\bibinfo{author}{Lenstra, Jr., H.W.}, \bibinfo{author}{Schoof, R.J.},
  \bibinfo{year}{1987}.
\newblock \bibinfo{title}{Primitive normal bases for finite fields}.
\newblock \bibinfo{journal}{Math. Comp.} \bibinfo{volume}{48},
  \bibinfo{pages}{217--231}.
\newblock \URLprefix \url{https://doi.org/10.2307/2007886},
  \DOIprefix\doi{10.2307/2007886}.
\bibitem[{Lidl and Niederreiter(1997)}]{Nieder}
\bibinfo{author}{Lidl, R.}, \bibinfo{author}{Niederreiter, H.},
  \bibinfo{year}{1997}.
\newblock \bibinfo{title}{Finite fields}. volume~\bibinfo{volume}{20}.
\newblock \bibinfo{edition}{Second} ed., \bibinfo{publisher}{Cambridge
  University Press, Cambridge}.
\bibitem[{Meletiou and Mullen(1992)}]{mullennote}
\bibinfo{author}{Meletiou, G.}, \bibinfo{author}{Mullen, G.L.},
  \bibinfo{year}{1992}.
\newblock \bibinfo{title}{A note on discrete logarithms in finite fields}.
\newblock \bibinfo{journal}{Applicable Algebra in Engineering, Communication
  and Computing} \bibinfo{volume}{3}, \bibinfo{pages}{75--78}.
\bibitem[{Mullen(2016)}]{mullConj}
\bibinfo{author}{Mullen, G.L.}, \bibinfo{year}{2016}.
\newblock \bibinfo{title}{Some open problems arising from my recent finite
  field research}, in: \bibinfo{booktitle}{Contemporary developments in finite
  fields and applications}. \bibinfo{publisher}{World Scientific}, pp.
  \bibinfo{pages}{254--269}.
\bibitem[{Rani et~al.(2021a)Rani, Sharma and Tiwari}]{Rani2}
\bibinfo{author}{Rani, M.}, \bibinfo{author}{Sharma, A.K.},
  \bibinfo{author}{Tiwari, S.K.}, \bibinfo{year}{2021}a.
\newblock \bibinfo{title}{On $r$-primitive $k$-normal elements over finite
  fields}.
\newblock \bibinfo{journal}{Communicated} .
\bibitem[{Rani et~al.(2021b)Rani, Sharma, Tiwari and Gupta}]{Rani1}
\bibinfo{author}{Rani, M.}, \bibinfo{author}{Sharma, A.K.},
  \bibinfo{author}{Tiwari, S.K.}, \bibinfo{author}{Gupta, I.},
  \bibinfo{year}{2021}b.
\newblock \bibinfo{title}{On the existence of pairs of primitive normal
  elements over finite fields}.
\newblock \bibinfo{journal}{S{\~a}o Paulo Journal of Mathematical Sciences} ,
  \bibinfo{pages}{1--18}.
\bibitem[{Reis(2019)}]{reis19}
\bibinfo{author}{Reis, L.}, \bibinfo{year}{2019}.
\newblock \bibinfo{title}{Existence results on $ k $-normal elements over
  finite fields}.
\newblock \bibinfo{journal}{Revista Matem{\'a}tica Iberoamericana}
  \bibinfo{volume}{35}, \bibinfo{pages}{805--822}.
\bibitem[{Reis and Thomson(2018)}]{reis18}
\bibinfo{author}{Reis, L.}, \bibinfo{author}{Thomson, D.},
  \bibinfo{year}{2018}.
\newblock \bibinfo{title}{Existence of primitive 1-normal elements in finite
  fields}.
\newblock \bibinfo{journal}{Finite Fields and Their Applications}
  \bibinfo{volume}{51}, \bibinfo{pages}{238--269}.
\bibitem[{Sharma et~al.(2022)Sharma, Rani and Tiwari}]{Avnish1}
\bibinfo{author}{Sharma, A.K.}, \bibinfo{author}{Rani, M.},
  \bibinfo{author}{Tiwari, S.K.}, \bibinfo{year}{2022}.
\newblock \bibinfo{title}{Primitive normal pairs with prescribed norm and
  trace}.
\newblock \bibinfo{journal}{Finite Fields and Their Applications}
  \bibinfo{volume}{78}, \bibinfo{pages}{101976}.
\bibitem[{{The Sage Developers}(2020)}]{Sage}
\bibinfo{author}{{The Sage Developers}}, \bibinfo{year}{2020}.
\newblock \bibinfo{title}{{S}ageMath, the {S}age {M}athematics {S}oftware
  {S}ystem ({V}ersion 9.0)}.
\newblock \URLprefix \url{https://www.sagemath.org}.
\bibitem[{Tian and Qi(2006)}]{TianQi}
\bibinfo{author}{Tian, T.}, \bibinfo{author}{Qi, W.F.}, \bibinfo{year}{2006}.
\newblock \bibinfo{title}{Primitive normal element and its inverse in finite
  fields}.
\newblock \bibinfo{journal}{Acta Math. Sinica (Chin. Ser.)}
  \bibinfo{volume}{49}, \bibinfo{pages}{657--668}.

\end{thebibliography}

\newpage
\appendix
\section{}\label{A1}
{\fontsize{6}{8}\selectfont{
	\begin{longtable}[h!]{p{2mm}|p{12cm}}
		\caption{\scriptsize Pairs $(q,n)$ for which \texttt{TEST_SIEVE(q,n)} returns \texttt{False} in the cases $n=$ 5, 6.\label{Table6}}\\
		\hline
		$n$&$q$\\
		
		\hline
		\endfirsthead
		\multicolumn{2}{l}{Continuation of Table \ref{Table6}}\\ 
		
		\hline
		$n$&$q$\\
		\hline
		\endhead
		
		\hline
		\endfoot
		
		\hline
		\endfoot
		5&2, 3, 4, 7, 8, 9, 11, 13, 16, 17, 19, 23, 27, 29, 31, 32, 37, 41, 43, 47, 49, 53, 59, 61, 64, 67, 71, 73, 79, 81, 83, 89, 97, 101, 103, 107, 109, 113, 121, 127, 128, 131, 137, 139, 149, 151, 157, 163, 167, 169, 173, 179, 181, 191, 193, 197, 199, 211, 223, 227, 229, 233, 239, 241, 243, 251, 256, 257, 263, 269, 271, 277, 281, 283, 289, 293, 307, 311, 313, 317, 331, 337, 343, 347, 349, 353, 359, 361, 367, 373, 379, 383, 389, 397, 401, 409, 419, 421, 431, 433, 439, 443, 449, 457, 461, 463, 467, 479, 487, 491, 499, 503, 509, 512, 521, 523, 529, 541, 547, 557, 563, 569, 571, 577, 587, 593, 599, 601, 607, 613, 617, 619, 631, 641, 643, 647, 653, 659, 661, 673, 677, 683, 691, 701, 709, 719, 727, 729, 733, 739, 743, 751, 757, 761, 769, 773, 787, 797, 809, 811, 821, 823, 827, 829, 839, 841, 853, 857, 859, 863, 877, 881, 883, 887, 907, 911, 919, 929, 937, 941, 947, 953, 961, 967, 971, 977, 991, 997, 1009, 1013, 1019, 1021, 1024, 1031, 1033, 1039, 1049, 1051, 1061, 1063, 1069, 1087, 1091, 1093, 1103, 1109, 1117, 1123, 1129, 1151, 1153, 1163, 1171, 1181, 1193, 1201, 1213, 1229, 1231, 1237, 1249, 1259, 1279, 1289, 1291, 1301, 1303, 1321, 1327, 1331, 1361, 1367, 1369, 1373, 1381, 1399, 1409, 1423, 1429,
		1439, 1447, 1451, 1453, 1459, 1471, 1481, 1483, 1489, 1499, 1511, 1531, 1543, 1549, 1559, 1567, 1571, 1579, 1583, 1597, 1601, 1609, 1619, 1621, 1627, 1667, 1669, 1681, 1693, 1697, 1699, 1709, 1721, 1723, 1741, 1747, 1753, 1759, 1777, 1783, 1789, 1801, 1811, 1831, 1849, 1861, 1867, 1871, 1873, 1879, 1901, 1931, 1933, 1949, 1951, 1993, 1999, 2003, 2011, 2017, 2029, 2053, 2081, 2083, 2087, 2089, 2111, 2113, 2129, 2131, 2137, 2141, 2143, 2161, 2179, 2197, 2203, 2209, 2221, 2239, 2251, 2269, 2281, 2293, 2297, 2311, 2341, 2347, 2351, 2371, 2377, 2381, 2389, 2401, 2411, 2437, 2441, 2467, 2473, 2503, 2521, 2531, 2539, 2551, 2557, 2591, 2593, 2621, 2647, 2659, 2671, 2683, 2689, 2699, 2707, 2711, 2713, 2719, 2731, 2741, 2749, 2767, 2791, 2797, 2801, 2803, 2809, 2851, 2857, 2861, 2887, 2927, 2953, 2971, 3001, 3011, 3019, 3037, 3041, 3049, 3061, 3067, 3079, 3109, 3121, 3169, 3181, 3191, 3217, 3221, 3251, 3259, 3271, 3301, 3319, 3331, 3361, 3371, 3391, 3433, 3457, 3461, 3463, 3469, 3481, 3491, 3499, 3511, 3529, 3541, 3547, 3571, 3581, 3613, 3631, 3671, 3691, 3697, 3701, 3721, 3727, 3733, 3739, 3761, 3821, 3823, 3851, 3853, 3877, 3881, 3907, 3911, 3919, 3931, 4001, 4003, 4019, 4021, 4051, 4057, 4096, 4111, 4129, 4159, 4201, 4211, 4219, 4229, 4231, 4241, 4243, 4261, 4271, 4327, 4391, 4421, 4441, 4447, 4451, 4481, 4489, 4519, 4561, 4591, 4603, 4621, 4651, 4691, 4751,
		4759, 4789, 4801, 4831, 4861, 4871, 4909, 4931, 4951, 4957, 4999, 5011, 5021, 5041, 5101, 5119, 5167, 5171, 5179, 5209, 5261, 5281, 5329, 5351, 5381, 5419, 5431, 5441, 5449, 5471, 5501, 5503, 5521, 5531, 5581, 5591, 5641, 5659, 5701, 5741, 5743, 5779, 5791, 5801, 5821, 5839, 5851, 5861, 5881, 5923, 5981, 6007, 6011, 6043, 6091, 6121, 6131, 6151, 6163, 6211, 6221, 6241, 6271, 6301, 6361, 6421, 6427, 6451, 6469, 6481, 6521, 6561, 6571, 6581, 6637, 6661, 6679, 6691, 6761, 6763, 6781, 6791, 6841, 6871, 6889, 6961, 6991, 7001, 7039, 7121, 7151, 7177, 7309, 7321, 7351, 7411, 7417, 7451, 7481, 7541, 7549, 7561, 7591, 7621, 7681, 7741, 7841, 7921, 7951, 8011, 8101, 8161, 8171, 8191, 8221, 8231, 8311, 8317, 8419, 8431, 8461, 8501, 8521, 8581, 8641, 8731, 8737, 8761, 8779, 8821, 8941, 8971, 9001, 9091, 9151, 9161, 9181, 9199, 9241, 9311, 9391, 9409, 9421, 9491, 9511, 9521, 9601, 9619, 9631, 9661, 9721, 9769, 9781, 9811, 9829, 9871, 9901, 9931, 10111, 10141, 10151, 10201, 10321, 10459, 10501, 10531, 10609, 10651, 10711, 10771, 10831, 10861, 10891, 11047, 11071, 11131, 11161, 11251, 11257, 11311, 11411, 11491, 11551, 11621, 11701, 11719, 11731, 11821, 11881, 11941, 11971, 12211, 12241, 12301, 12391, 12421, 12433, 12451, 12511, 12541, 12601, 12721, 12769, 12781, 12841, 13171, 13291, 13381, 13399, 13411, 13441, 13567, 13681, 13711, 13831, 13921, 14071, 14221, 14251, 14281, 14341, 14401, 14431, 14551, 14641, 14731, 14771, 14821, 14851, 15031, 15121, 15271, 15331, 15361, 15391, 15451, 15511, 15541, 15601, 15661, 15901, 15991, 16111, 16141, 16231, 16339, 16381, 16651, 16831, 16921, 17011, 17161, 17191, 17341, 17431, 17491, 17581, 17791, 17851, 17911, 18061, 18121, 18181, 18301, 18451, 18481, 18661, 18691, 19081, 19141, 19231, 19321, 19381, 19441, 19471, 19501, 19531, 19801, 19891, 20011, 20101, 20161, 20431, 20551, 20641, 20749, 21001, 21121, 21211, 21481, 21751, 21841, 21871, 21961, 22051, 22111, 22201, 22291, 22441, 22531, 22621, 22741, 22801, 22861, 23011, 23071, 23131, 23311, 23371, 23431, 23761, 23971, 24061, 24091, 24121, 24151, 24181, 24391, 24421, 24481, 24571, 24781, 24841, 25111, 25261, 25411, 25621, 25741, 26041, 26251, 26701, 26821, 26881, 27031, 27091, 27361, 27691, 27751, 27901, 28051, 28081, 28351, 28561, 28711, 28771, 29131, 29191, 29401, 29581, 29611, 29641, 29671, 29881, 30211, 30241, 30391, 30661, 30871, 31081, 31321, 31531, 31981, 32191, 32341, 32371, 32761, 32971, 33091, 33181, 33391, 33811, 34171, 34651, 35281, 35311, 35491, 36061, 36481, 36541, 37171, 37591, 37951, 38611, 38851, 39901, 40111, 40531, 41611, 42331, 42841, 43261, 43891, 44521, 46411, 47881, 51871, 53551, 57121, 60901, 62791, 63211, 65521, 70981 \\
		\hline
		6&5, 7, 11, 13, 17, 19, 23, 25, 29, 31, 37, 41, 43, 47, 49, 53, 59, 61, 67, 71, 73, 79, 83, 89, 97, 101, 103, 107, 109, 113, 121, 127, 131, 137, 139, 149, 151, 157, 163, 169, 179, 181, 191, 199, 211, 229, 241.	
	\end{longtable}
}}
\begin{algorithm}[h]
			{\scriptsize
				\caption{\small Testing the sieving inequality for the existence of primitive $1$-normal pair $(\alpha,\alpha^{-1})$}
				\begin{algorithmic}[1]
					\Require $q$, $n$, $\vartheta$
					\Ensure \texttt{True}, if $q^{n/2-\vartheta}>2W(h)W(d)W(H)\mathcal{S}$ for some $h\mid x^n-1$, $d\mid R$ and $H\mid G$ otherwise, \texttt{False}.
					\Procedure{ \texttt{test_sieve}}{\texttt{q,n},$\vartheta$}
					\State Construct the set $A$ of divisors of $\mathrm{rad}(q^n-1)$ and  set $B$ of divisors of $\frac{\mathrm{rad}(x^n-1)}{x-1}$.
					\For{$d$ in $A$}\For{$H$ in $B$}
					\State $h=H$
					\State Construct the set $L_1$ of prime divisors $\frac{\mathrm{rad}(q^n-1)}{d}$
					\State Construct the set $L_2$ of irreducible factors of $\frac{\mathrm{rad}(x^n-1)}{H(x-1)}$ and the set $L_3=L_2\cup\{x-1\}$
					\State Compute $\mathcal{D}=1-\sum\limits_{p_1 \in L_1}\frac{1}{p_1}-\sum\limits_{p_2 \in L_2}\frac{1}{q^{\mathrm{deg}(p_2)}}-\sum\limits_{p_3 \in L_3}\frac{1}{q^{\mathrm{deg}(p_3)}}$
					\If{$\mathcal{D}>0$}
					\State compute $\mathcal{S}=\frac{\ |L_1|+|L_2|+|L_3|-1}{\mathcal{D}}+2$
					\If{$q^{n/2-\vartheta}>2W(d) W(H)^2\mathcal{S}$}\Comment{Sieving inequality}
					\State \textbf{return} \texttt{True} 		
					\EndIf
					\EndIf
					\EndFor
					\EndFor
					\State \textbf{return} \texttt{False}
					\EndProcedure
			\end{algorithmic}}
		\end{algorithm}
		\begin{algorithm}[t!]
			{\scriptsize
				\caption{\small Explicit search for a primitive $1$-normal pair $(\alpha,\alpha^{-1})$}
				\begin{algorithmic}[1]
					\Require $q$, $n$
					\Ensure \texttt{True}, if a pair $(\alpha,\alpha^{-1})$ of primitive $1$-normal elements exists in $\mathbb{F}_{q^n}$ over $\mathbb{F}_q$, otherwise, \texttt{False}.
					\Procedure{\texttt{direct_search}}{\texttt{q,n}}
					\For{$\beta$ in $\mathbb{F}_{q^n}$}
					\State $\alpha=\beta^{q}-\beta$ 
					\If{$\alpha\neq 0$} 
					\State construct $m_{\alpha}(x)$ and $m_{\alpha^{-1}}(x)$ \Comment{ $m_{\alpha}(x)=\sum_{i=0}^{n-1}\alpha^{q^i}x^{n-i-1}$}
					\If{$\mathrm{deg}(\mathrm{gcd}(m_\alpha(x),x^n-1))=1$ \& $\mathrm{deg}(\mathrm{gcd}(m_{\alpha^{-1}}(x),x^n-1))=1$ \& $\mathrm{ord}(\alpha)=q^n-1$}
					\State \textbf{return} \texttt{True}
					\EndIf
					\EndIf
					\EndFor
					\State \textbf{return} \texttt{False}
					\EndProcedure
			\end{algorithmic}}
		\end{algorithm}
\end{document}